\documentclass[11pt]{article}

\usepackage[utf8]{inputenc}
\usepackage[T1]{fontenc}
\usepackage{amsmath,amssymb,amsthm}
\usepackage[margin=1.2in]{geometry}
\usepackage{hyperref}

\newtheorem{theorem}{Theorem}[section]
\newtheorem{lemma}[theorem]{Lemma}
\newtheorem{proposition}[theorem]{Proposition}
\newtheorem{corollary}[theorem]{Corollary}
\newtheorem{problem}[theorem]{Open Problem}
\theoremstyle{remark}
\newtheorem{remark}[theorem]{Remark}
\theoremstyle{definition}
\newtheorem{definition}[theorem]{Definition}

\theoremstyle{plain}
\newtheorem*{thmA}{Theorem A}

\newcommand{\ACA}{\mathsf{ACA}_0}
\newcommand{\RCA}{\mathsf{RCA}_0}
\newcommand{\RCAs}{{\mathsf{RCA}_0^*}}
\newcommand{\WKLs}{{\mathsf{WKL}_0^*}}
\newcommand{\WKL}{\mathsf{WKL}}
\newcommand{\ISig}{\mathrm{I}\Sigma^0_1}
\newcommand{\Coll}[1]{\Sigma^1_{#1}\text{-}\mathsf{Coll}}
\newcommand{\DC}[1]{\Sigma^1_{#1}\text{-}\mathsf{DC}}
\newcommand{\Ztwo}{\mathsf{Z}_2}
\newcommand{\ZFCm}{\mathsf{ZFC}^{-}}
\newcommand{\N}{\mathbb{N}}
\newcommand{\Rfns}{\mathsf{Rfn}}
\newcommand{\Rfn}[1]{\Sigma^1_{#1}\text{-}\mathsf{Rfn}}
\newcommand{\inack}{\in_{\textsf{ack}}}

\title{L\'evy--Montague reflection is\\ $\Pi^1_1$-conservative over $\WKL_0$}
\author{Fedor Pakhomov\thanks{Department of Mathematics: Analysis, Logic
and Discrete Mathematics, Ghent University, Krijgslaan 281, B-9000
Ghent, Belgium; and Steklov Mathematical Institute of the Russian
Academy of Sciences, Moscow. \texttt{fedor.pakhomov@ugent.be}. The
work of the author was funded by the FWO grant G0F8421N.}}
\date{August 2026}

\begin{document}
\maketitle

\begin{abstract}
We study a L\'evy--Montague reflection scheme $\Rfns$ in second-order
arithmetic: for each formula $\varphi$, the scheme asserts that every
set belongs to a countable coded $\omega$-model such that $\varphi$ is
absolute, at all parameters from the model, between the model and the
universe. Our central result is a model extension construction: every
countable model of $\RCA$ can be extended, without changing its
first-order part, to a model of $\WKL_0$ together with the full
scheme $\Rfns$. It follows at once that $\WKL_0+\Rfns$ is
$\Pi^1_1$-conservative over both $\WKL_0$ and $\RCA$, that its
first-order part is exactly $\mathrm{I}\Sigma_1$, and that it is $\Pi^0_2$-conservative over $\mathsf{PRA}$.
The result opens an avenue for adopting, within a theory conservative over
$\mathsf{PRA}$, Feferman's $\mathsf{ZFC}$-formalization of universe-based category-theoretic
arguments that was achieved using L\'evy--Montague reflection.
The conservation proof itself, however, is non-finitary. The extension
is the union of an $\omega_1$-tower of forcing extensions, and its
uncountable cofinality is what secures reflection. We are only able
to prove the conservation in $\mathsf{PRA}+\textsf{1-Con}(\mathsf{Z}_2)$. The
results were obtained with extensive use of Anthropic's large language model Fable~5.
\end{abstract}



\section{Introduction}\label{sec:intro}

In well-known work, L\'evy and Montague \cite{Levy1960,Montague1961}
established that truth in the set-theoretic universe reflects to the
levels $V_\alpha$ of the cumulative hierarchy. The literature contains
several versions of this reflection principle. The strongest, and the
one relevant to the present work, is the reflection scheme with
absoluteness: for each formula $\varphi(\vec x)$ of the language of set
theory, $\mathsf{ZF}$ proves that there are arbitrarily large ordinals
$\alpha$ such that
\[
\forall\vec x\in V_\alpha\,\bigl(\varphi(\vec x)\leftrightarrow
\varphi^{V_\alpha}(\vec x)\bigr),
\]
where $\varphi^{V_\alpha}$ is the relativization of $\varphi$ to
$V_\alpha$.

This paper studies the natural adaptation of this scheme to
second-order arithmetic, with coded $\omega$-models taking the role of the
ranks $V_\alpha$. Working within systems of second-order arithmetic, a
\emph{coded $\omega$-model} is a set $Y$, regarded
as coding the $L_2$ structure whose sets are the sections $(Y)_y$ of
$Y$ and whose first-order part consists of the naturals of the ambient model.
Writing $X\dot\in Y$ for ``$X$ is a section of $Y$'' and $\varphi^Y$
for the relativization of the set quantifiers of $\varphi$ to the
sections of $Y$, the reflection scheme $\Rfns$ is
\[
\forall\vec P\,\exists Y\,\Bigl(\vec P\dot\in Y\ \wedge\
\forall\vec x\,\forall\vec X\dot\in Y\;
\bigl(\varphi(\vec P,\vec x,\vec X)\leftrightarrow
\varphi^Y(\vec P,\vec x,\vec X)\bigr)\Bigr):
\]
every tuple of parameters lies in a coded $\omega$-model such that
$\varphi$ is absolute, at all parameters from the model, between the
model and the universe. The instances with $\varphi$ restricted to 
$\Sigma^1_n$ formulas form the scheme $\Rfn{n}$.

Over a strong enough base the scheme is essentially known. Over
$\ACA$, where there are universal $\Sigma^1_n$-formulas, $\Rfn{n}$ is
straightforwardly equivalent to the statement that every set is contained
in a countable coded $\beta_n$-model. Here the $\beta_n$-models are the
$\omega$-models that are $\Sigma^1_n$-elementary submodels of the universe.
The $\beta_n$-reflection principles of this type have been studied
by Simpson in \cite[\S VII.7]{Simpson}. In particular, from the results there
it is immediate that $\Rfn{n}$ is equivalent to strong $\Sigma^1_n$
dependent choice (we give a proof of this in Theorem~\ref{thm:acalr}).
The $\Sigma^1_n$ choice and dependent choice schemes themselves go back
to Kreisel \cite{Kreisel1962,Kreisel1968}. The \emph{strong} dependent
choice scheme used here, the form with an exact equivalence to
$\Rfn{n}$, was introduced by Simpson \cite[\S VII.6]{Simpson}. For $n=1,2$
strong $\Sigma^1_n$ dependent choice is equivalent to $\Pi^1_n\text{-}\mathsf{CA}_0$
\cite[Theorem~VII.6.9]{Simpson}, see also \cite{Friedman1970}.
In general, however, strong
$\Sigma^1_n$ dependent choice is only $\Pi^1_4$-conservative over
$\Pi^1_n\text{-}\mathsf{CA}_0$ \cite[Theorem~VII.6.20]{Simpson}.
The plain choice schemes, for their part, carry classical
conservation theorems: $\Sigma^1_1\text{-}\mathsf{AC}_0$
is $\Pi^1_2$-conservative over
$\ACA$, by the recursively saturated model argument of Barwise and
Schlipf \cite{BarwiseSchlipf}, see also \cite[\S IX.4]{Simpson}; via
a similar construction it is easy to prove $\Pi^1_{n+2}$-conservativity of
$\Sigma^1_{n+1}\text{-}\mathsf{AC}_0$ over
$\Pi^1_{n}\text{-}\mathsf{CA}_0+\Sigma^1_n\text{-}\mathsf{AC}_0$, for $n\ge 1$.
A weaker form of the $\omega$-model reflection scheme is the one
asserting that for every true formula with fixed values of the parameters
there is an $\omega$-model, containing those parameters, in which the
formula holds. Over $\mathsf{ACA}_0$, full $\omega$-model reflection
is equivalent to the scheme of
bar induction, a classical theorem of Friedman \cite{Friedman1975}; see
also \cite[\S VIII.5]{Simpson}.
J\"ager and Strahm \cite{JagerStrahm} refined this to a level-by-level
correspondence between $\Pi^1_{n+2}$-$\omega$-model reflection and
$\Pi^1_n$ bar induction, and Pakhomov and Walsh \cite{PakhomovWalsh}
reduced $\omega$-model reflection over $\mathsf{ACA}_0$ to iterated syntactic reflection.

For the unsound consistent theory $\WKLs+\neg\ISig$ the scheme $\Rfns$ is simply provable
(Theorem~\ref{thm:wkllr}). This is a
quick consequence of the isomorphism theorem of Fiori-Carones,
Ko{\l}odziejczyk, Wong and Yokoyama \cite{FKWY}, by which truth
in a model of $\WKLs+\neg\ISig$ collapses into every coded
$\omega$-model of that theory. In that setting collection holds for all
$L_2$ formulas.

The heart of the paper is the intermediate regime, with the base theories 
being $\RCA$ or $\WKL_0$, which have $\Sigma^0_1$ induction but no
arithmetical comprehension. There the full scheme is unprovable
(Corollary~\ref{cor:unif}). However, curiously, $\WKL_0$ proves that
every set is contained in a countable coded \emph{strict $\beta$-model},
a coded $\omega$-model elementary
for $\Sigma^1_1$ formulas with $\Pi^0_1$ matrix; over $\WKL_0$, this
elementarity is moreover equivalent to satisfying $\WKL_0$. In this
precise form the result is due to Simpson
\cite[Theorems~VIII.2.2 and~VIII.2.6]{Simpson}; the phenomenon goes
back to Scott and Tennenbaum \cite{ScottTennenbaum}.

\begin{thmA}[Theorem~\ref{thm:lr}]
$\WKL_0+\Rfns$ is $\Pi^1_1$-conservative over $\RCA$; hence its
first-order part is exactly the set of
$\mathrm{I}\Sigma_1$-provable sentences.
\end{thmA}

In particular the reflection scheme is $\Pi^1_1$-conservative over
$\WKL_0$ itself: a $\Pi^1_1$ theorem of
$\WKL_0+\Rfns$ is a theorem of $\RCA$, and a fortiori of $\WKL_0$.

Theorem~A both strengthens and borrows from Harrington's theorem that
$\WKL_0$ is $\Pi^1_1$-conservative over $\RCA$
\cite[\S IX.2]{Simpson}.  Since $\RCA$ is
$\Pi^0_2$-conservative over $\mathsf{PRA}$ \cite[\S IX.3]{Simpson},
Theorem~A also yields the $\Pi^0_2$-conservativity of $\WKL_0+\Rfns$
over $\mathsf{PRA}$ (Corollary~\ref{cor:pra}).

The comparison across bases thus reverses at each end, as the following
table summarizes.

\begin{center}
\begin{tabular}{l|l}
base & reflection scheme $\Rfns$\\
\hline
$\WKLs+\neg\ISig$ & provable (one witness for all formulas)\\[2pt]
$\RCA$, $\WKL_0$ & $\Pi^1_1$- but not $\Pi^1_2$-conservative; unprovable\\[2pt]
$\ACA$ & $\Rfn{n}\equiv$ strong $\Sigma^1_n$-DC, level by level\\
\end{tabular}
\end{center}

It is important to emphasize the connection of our conservation result
to the revised Hilbert program. Hilbert's original program asked for a finitistic
justification of infinitistic mathematics. Although G\"odel's incompleteness
theorems rule this out wholesale, they do not prevent partial results.
In the revision articulated by Simpson \cite{Simpson1988}, one adopts Tait's
identification of finitistic reasoning with $\mathsf{PRA}$
\cite{Tait1981} and asks which portions of infinitistic mathematics
are \emph{finitistically reducible}, that is, conservative over
$\mathsf{PRA}$ for $\Pi^0_2$ sentences. On such a portion the
infinitistic apparatus is a convenience that finitism itself can
certify. The
showcase of the program is $\WKL_0$. A substantial part of ordinary
mathematics is provable in it: the Heine--Borel covering lemma, the
Hahn--Banach theorem for separable spaces, the G\"odel completeness
theorem, and much else \cite{Simpson1988,Simpson}. At the same time, by
a theorem of Friedman, $\WKL_0$ is $\Pi^0_2$-conservative over
$\mathsf{PRA}$ \cite[\S IX.3]{Simpson}. The conservativity of $\WKL_0$
is what carries the foundational weight: it converts every
$\WKL_0$-proof of a $\Pi^0_2$ sentence, however infinitistic in
appearance, into a guarantee that the sentence is already a theorem of
$\mathsf{PRA}$.

Theorem~A places $\WKL_0+\Rfns$ in the same finitistically reducible
zone, and the addition of $\Rfns$ is not a vacuous one. 
It is unprovable not only in $\WKL_0$ but even in $\Pi^1_n\text{-CA}_0$, because
$\Rfn{n+1}$ implies $\Pi^1_{n+1}\text{-CA}_0$ over $\ACA$
(Theorem~\ref{thm:acalr} with \cite[\S VII.6]{Simpson}), and its new
consequences appear already at
the $\Sigma^1_1$ level (Corollary~\ref{cor:unif}).

Does the added strength meaningfully expand $\WKL_0$ as a framework
for mathematics, rather than merely exceed it in logical strength? We
have not examined this in any detail, but there is a concrete reason
to expect so. Reflective universes are the device by which Feferman
\cite{Feferman1969} gave conservative set-theoretic foundations for
the kind of category-theoretic reasoning that normally uses the axiom
of Grothendieck universes --- a typical application of this is
to embed what would otherwise be large categories as objects in other
categories. More precisely, he adjoined to $\mathsf{ZFC}$
a constant for a set $U$ with an axiom scheme asserting that $U$ reflects
the universe. Then the idea is that in order to prove facts about
large totalities such as the category of all groups it suffices to prove the
counterpart facts about the category of all $U$-groups, which is a set-sized
object. Theorem~A supplies a very similar reflection principle
over a finitistically reducible base. And Feferman's further reasoning
with the conservative extension by a constant $U$ can be replicated
(this is a straightforward argument based on the compactness theorem).
Though, of course, $\mathsf{WKL}_0+\Rfns$ is a far weaker theory than $\mathsf{ZFC}$ in many other important respects. Thus, we anticipate that when
applying this approach to the formalization of particular
sufficiently complex mathematical
arguments, extra difficulties will often surface,
related to uses of non-recursive comprehension,
or perhaps some uses of choice principles (beyond what $\mathsf{WKL}_0$
delivers), as well as the power-set axiom.

The proof of Theorem~A is a model extension construction: the
second-order part of a countable model of $\RCA$ is extended, keeping
the first-order part fixed, by an \emph{$\omega_1$-tower of forcing
extensions}. This is an increasing continuous $\omega_1$-chain of countable
sets $\mathcal S_\alpha$ of subsets of the fixed first-order part,
where for each $\alpha<\omega_1$ there is a set $X_\alpha\in \mathcal{S}_{\alpha+1}$
internally coding $\mathcal{S}_\alpha$ ($X_\alpha$ considered as
a code for an internally countable family of sets of internal naturals
enumerates precisely $\mathcal{S}_\alpha$).
This construction of the chain easily guarantees the validity
of the reflection principle in its union.

An important note here is about the meta-theoretic assumptions required to prove
this conservation result. While the assertion of conservativity
is a $\Pi^0_2$ statement, the argument passes through an uncountable object and uses
its uncountable cofinality essentially. Section~\ref{sec:onecon}
bounds the metamathematics: $\mathsf{PRA}$ together with the
$1$-consistency of full second-order arithmetic suffices.
Thus it remains open whether the conservation result of this
paper can be proved by finitistic (or at least relatively modest) means.
We discuss this further at the end of Section~\ref{sec:onecon}.

Theorem~A is sharp: the conservation does not extend to
$\Sigma^1_1$. Reflection refutes the \emph{uniform
enumerability} of the universe, the existence of a single formula
naming every set of the model by a number. Uniform enumerability by an
arithmetical formula is a $\Pi^1_1$ property, and it holds in the low
$\omega$-models of $\WKL_0$. So already $\Rfn{1}$ proves a
$\Sigma^1_1$ sentence unprovable in $\WKL_0$ (Theorem~\ref{thm:unif},
Corollary~\ref{cor:unif}).

\subsection{The use of AI in the research}
This paper grew out of an extended interaction between the author and
Anthropic's language model Fable~5, and it seems right to record how.
The proof of the central result is not technically demanding and the
real obstruction to finding it was rather that the statement looks too good
to be true. The author, personally, was not at all expecting a principle like that to
be conservative over $\WKL_0$ or $\RCA$.

In fact the starting problem for the AI-driven investigation
was to calibrate the strength of the collection scheme over
theories like $\RCA$ and $\WKL_0$. The author had long been aware
of the situation with the collection scheme over $\ACA$,
where it is level-by-level equivalent to the choice scheme $\mathsf{AC}$,
and of the fact that the collection scheme is provable in
$\WKLs+\neg\ISig$ using the results of \cite{FKWY}.

The route to the result is best described as genuinely collaborative:
the model made many investigative attempts in various directions,
proposing various constructions, while the human author supplied insight, direction,
and the repeated recontextualization of what the model proposed. Probably the
single most important contribution on the model's side was the idea of
building models of $\Sigma^1_1$-collection as $\omega_1$-chains of
extensions of the second-order part, each stage containing a universal
set for its predecessor. In retrospect this is unsurprising, in the
light of the known validity of collection in $\omega_1$-like nonstandard models of
arithmetic \cite{MillsParis,EnayatMohsenipour}.
The realization that the very same construction yields the much
stronger \emph{reflection} scheme, rather than $\Sigma^1_1$-collection alone, is due
to the author. The proof of the main technical
Lemma~\ref{lem:codingns} in its current form is due to the human author:
although the core high-level idea of the relevance of adding universal sets
by forcing arose as a result of a collaborative interaction,
the AI system was not able to correctly
deal with the technical nuances of the argument in an autonomous manner.
It is also worth mentioning that the observation of Corollary~\ref{cor:weakcons} as a consequence of the
central result of Section~\ref{sec:towers} and Section~\ref{sec:wkl} is entirely due to
the AI system.

As for the text: it is almost entirely written by AI models, but a
lot of granular, low-level editing requests have been made to bring
it to the author's own standards and preferences. The human author has also
carefully inspected all the content of the article, to
ensure that the standard expectations of quality, as well as his
personal standards, are met.

\subsection{Organization}
Section~\ref{sec:prelim} fixes notation, defines the reflection scheme,
and derives collection from reflection. Section~\ref{sec:aca} treats the
base $\ACA$ and the equivalence with strong dependent choice, and
Section~\ref{sec:wkl} treats the base $\WKLs+\neg\ISig$. These are the
two known endpoints. Section~\ref{sec:towers} builds the $\omega_1$-towers,
proves the central conservativity result, Theorem~A, extends the
conservation to the weak bases $\RCAs$ and $\WKLs$, and poses the open
problem of running the tower over them directly.
Section~\ref{sec:optimality} shows that the conservation is optimal:
it does not extend to $\Sigma^1_1$ sentences. Section~\ref{sec:onecon}
bounds the metamathematics of the proof of Theorem~A, ending with the
question of finitism.

\section{The reflection scheme}\label{sec:prelim}

We work in the two-sorted language $L_2$ of second-order arithmetic with
constants $0,1$, addition and multiplication function symbols and
a unary function symbol $\mathsf{exp}(x)$ for base-two exponentiation.
The classes $\Sigma^1_0=\Pi^1_0=\Pi^0_\infty$ (formulas without second-order quantifiers), $\Sigma^1_n$, $\Pi^1_n$ (prenex classes with second-order quantifiers and $\Pi^0_\infty$ matrices), $\Delta^0_0$ (formulas without second-order quantifiers and only bounded first-order quantifiers), and $\Sigma^0_n$, $\Pi^0_n$ (classes of formulas with first-order prenexes and $\Delta^0_0$-matrices) are defined in a standard manner using our signature.
For first-order languages --- plain arithmetic, or arithmetic
augmented by a set constant $X$, as in the forcing construction
below --- we write the corresponding classes with a single index:
$\Delta_0$, $\Sigma_n$, $\Pi_n$, and $\Delta_0(X)$, $\Sigma_n(X)$,
$\Pi_n(X)$, etc.; the double-indexed classes always refer to the
two-sorted language $L_2$.
As usual we denote by $\underline{n}$ the numerals for natural numbers $n$, i.e.\ the closed terms $\underline{0}=0$ and $\underline{n+1}=(\underline{n}+1)$.
The weakest theory that we consider
is the Simpson--Smith
\cite{SimpsonSmith} system $\mathsf{RCA}_0^*$ that is axiomatized over
Robinson's arithmetic $Q$ with the extra recursive defining axioms for $\mathsf{exp}(x)$
by the schemes of $\Delta^0_0$-induction and $\Delta^0_1$-comprehension (comprehension for properties
that are both $\Sigma^0_1$ and $\Pi^0_1$).

We fix a monotone primitive recursive pairing bijection
$\langle\cdot,\cdot\rangle\colon\N^2\to\N$ and write
$(X)_i=\{m:\langle i,m\rangle\in X\}$ for the $i$-th
section of $X$. All schemes apply
to formulas with arbitrary parameters.


We code $\omega$-models by their list of sections and relativize to it:

\begin{definition}\label{def:cl}
Write $X\dot\in Y$ for $\exists z\,((Y)_z=X)$, read ``$X$ is a section of
$Y$''. For an $L_2$ formula $\varphi$ write $\varphi^Y$ for the result of
relativizing all set quantifiers of $\varphi$ to the sections of $Y$;
when the free set variables are instantiated by sections, $\varphi^Y$ is
arithmetical with parameter $Y$. For a theory $T$ we write $T^Y$ for
the relativized assertion that the sections of $Y$ form a model of $T$.
We say \emph{$Y$ codes an $\omega$-model over $P_1,\dots,P_k$} if each
$P_i\dot\in Y$ (any finite number of parameters, including none). The
structure coded is the $L_2$ structure whose first-order part is that
of the ambient model and whose sets are the sections of $Y$. 
\end{definition}

\begin{definition}\label{def:lr}
For an $L_2$ formula $\varphi(\vec P,\vec x,\vec X)$, with set variables
$\vec P,\vec X$ and number variables $\vec x$, the reflection instance
$\Rfns_\varphi$ is the sentence
\[
\forall\vec P\,\exists Y\,\Bigl(\vec P\dot\in Y\ \wedge\
\forall\vec x\,\forall\vec X\dot\in Y\;
\bigl(\varphi(\vec P,\vec x,\vec X)\leftrightarrow
\varphi^Y(\vec P,\vec x,\vec X)\bigr)\Bigr),
\]
where $\vec P\dot\in Y$ abbreviates $\bigwedge_i P_i\dot\in Y$, the
parameters $\vec P$ form a tuple of arbitrary length, and the inner
$\forall\vec X\dot\in Y$ ranges over sections of $Y$ (so that
$\varphi^Y$, evaluated at sections, is arithmetical in $Y$). Thus $Y$
need only \emph{contain} the chosen $\vec P$ and \emph{agree} with the
ambient model on its own sections. The scheme $\Rfn{n}$ collects
$\Rfns_\varphi$ for all $\Sigma^1_n$ and $\Pi^1_n$ formulas $\varphi$;
the full \emph{reflection scheme} $\Rfns$ is the union
of the $\Rfn{n}$.
\end{definition}

For a finite set of formulas $\Phi=\{\varphi_i(\vec{P},\vec{x}_i,\vec{X}_i)\mid i<k\}$ we write $\Rfns_\Phi$ for the \emph{simultaneous} statement: one
$Y$ with $\vec P\dot\in Y$ that is $\varphi$-elementary for every
$\varphi\in\Phi$. Since reflecting $\varphi$ is the same as reflecting
$\neg\varphi$, if all $\varphi_i\in \Sigma^1_n \cup \Pi^1_n$ then by
switching some of them with their negations, we may assume that all of them
are in $\Sigma^1_n$. Also we may assume that the variable vectors $\vec{x}_i,\vec{X}_i$
are pairwise disjoint.
Then \[\psi(\vec{P},\vec{x}_0,\ldots,\vec{x}_{k-1},\vec{X}_0,\ldots,\vec{X}_{k-1},y):=\bigwedge_{i<k}(y=\underline{i}\to\varphi_i)\] is
logically equivalent to a $\Sigma^1_n$ formula, and $\Rfns_\psi$, read off at
$y=\underline{0},\dots,\underline{k-1}$, yields $\Rfns_\Phi$ at a single $Y$. Thus
$\Rfn{n}\vdash\Rfns_\Phi$ for every finite $\Phi$ consisting of $\Sigma^1_n$-formulas.

The scheme asks $Y$ only to contain $\vec P$ and to be elementary; it
does not ask the sections to model anything. It implies, however, its
own strengthened variant in which the model reflected on is moreover
required to satisfy any prescribed true $\Pi^1_{n+1}$ sentences.
Indeed, let $\pi$ be a $\Pi^1_{n+1}$ sentence, written as
$\forall\vec X\,\sigma(\vec X)$ with $\sigma\in\Sigma^1_n$, and let
$\Phi$ be a finite set of $\Sigma^1_n$ formulas. The instance
$\Rfns_{\Phi\cup\{\sigma\}}$ delivers a $Y$ over $\vec P$ that is both
$\Phi$-elementary and $\sigma$-elementary. If $\pi$ is true, then
$\sigma$ holds at every tuple of sections of $Y$, so by
$\sigma$-elementarity $\sigma^Y$ holds at every tuple of sections ---
and that is exactly $\pi^Y$. Thus $\Rfn{n}$ proves
\[
\pi\ \to\ \forall\vec P\,\exists Y\,\bigl(\vec P\dot\in Y\ \wedge\
\text{$Y$ is $\Phi$-elementary}\ \wedge\ \pi^Y\bigr).
\]
In particular, already $\Rfn{1}$ makes any given true $\Pi^1_2$
sentence hold in the models on which we reflect; hence any theory
axiomatized by finitely many $\Pi^1_2$ sentences, once it holds in
the ambient model, may be assumed to hold in every $Y$ the scheme
delivers. Likewise, $\Rfn{n}$ secures any theory axiomatized by
finitely many true $\Pi^1_{n+1}$ sentences. This covers $\RCAs$,
$\RCA$, $\WKLs$, $\WKL_0$ and $\ACA$: each of these theories is
finitely axiomatizable by $\Pi^1_2$-sentences (for $\RCA$, $\WKL_0$ and $\ACA$ see
\cite[\S\S VIII.1--VIII.2]{Simpson}; the same
universal-$\Sigma^0_1$-formula arguments work for $\RCAs$ and
$\WKLs$, with the induction axioms handled by the finite
axiomatizability of $\mathrm{I}\Delta_0(\mathsf{exp})=\mathsf{EA}$
\cite{GaifmanDimitracopoulos}). Hence, whichever of them serves as the base
theory, we may always assume that $Y$ satisfies it. In particular,
if the base theory proves $\varphi\leftrightarrow\varphi'$, the
equivalence relativizes to any such $Y$, so a $Y$ delivered for
$\varphi'$ is then $\varphi$-elementary as well.

As in set theory, reflection proves collection. The scheme $\Coll{n}$
asserts, for $\Sigma^1_n$ formulas $\varphi(x,Y)$ perhaps with extra free variables:
\[\forall x\,\exists Y\,\varphi(x,Y)\to\exists X\,\forall x\,\exists
y\,\varphi(x,(X)_y).\]

\begin{proposition}\label{prop:rfncoll}
Over $\RCAs$, each instance $\Rfns_\Phi$ implies, for every $\varphi(x,Z)$
with $\varphi,\exists Z\,\varphi\in\Phi$, the strong collection statement
$\exists X\,\forall x\,(\exists Z\,\varphi(x,Z)\to\exists z\,\varphi(x,(X)_z))$.
In particular $\Rfn{n}$ implies $\Coll{n}$ for every $n\ge1$, and $\Rfns$
implies collection for all $L_2$ formulas.
\end{proposition}

\begin{proof}
Let $\vec Q$ be the parameters of $\varphi$ and apply $\Rfns_\Phi$ with
parameters $\vec Q$, obtaining $Y_0$ with the $\vec Q$ among its
sections. Fix $a$ with
$\exists Z\,\varphi(a,Z,\vec Q)$.
By elementarity for $\exists Z\,\varphi$ there is $z$ with
$\varphi^{Y_0}(a,(Y_0)_z,\vec Q)$, and by elementarity for $\varphi$
this gives $\varphi(a,(Y_0)_z,\vec Q)$.
For the level count: by contraction of like set quantifiers,
$\exists Z\,\varphi$ is $\RCAs$-provably equivalent to a $\Sigma^1_n$
formula, and $Y_0$ may be assumed to satisfy $\RCAs$, as above, so
elementarity for that formula yields elementarity for
$\exists Z\,\varphi$.
\end{proof}

\section{Reflection over $\ACA$: reflection is dependent
choice}\label{sec:aca}

For an $L_2$ formula $\eta(n,X,Y)$ write
$(Z)_{<n}:=\{\langle i,m\rangle\in Z : i<n\}$ for the restriction
of $Z$ to its columns below $n$. The \emph{strong} dependent choice scheme
strong $\Gamma\text{-}\mathsf{DC}$ consists, for $\eta\in\Gamma$, of the
statements
\[
\exists Z\,\forall n\,\bigl(\exists
Y\,\eta\bigl(n,(Z)_{<n},Y\bigr)\to
\eta\bigl(n,(Z)_{<n},(Z)_n\bigr)\bigr);
\]
strong $\DC{k}$ abbreviates strong $\Sigma^1_k\text{-}\mathsf{DC}$. 
Trivially, over $\ACA$ this is equivalent to the formulation from
\cite[\S VII.6]{Simpson}.

\begin{theorem}\label{thm:acalr}
Over $\ACA$, for every $n\ge1$ the scheme $\Rfn{n}$ is equivalent to
strong $\DC{n}$.
\end{theorem}

\begin{proof}
\emph{Reflection $\Rightarrow$ dependent choice.} Fix $\eta\in\Sigma^1_n$
with set parameters $\vec Q$; by contraction of like set quantifiers,
$\exists Y\,\eta$ is $\ACA$-provably equivalent to a $\Sigma^1_n$
formula $\chi$. Apply $\Rfns_\Phi$ with parameters $\vec Q$, where
$\Phi=\{\eta,\chi\}$. This yields $Y_0$ with $\vec Q$ among its
sections and with elementarity for $\eta$ and $\chi$. Taking the
comprehension instances along, as after Definition~\ref{def:lr}, we
may assume $\ACA^{Y_0}$: then $\emptyset$ is a section, every set
$\Delta^0_1$-definable from sections is again a section, and the
contraction equivalence relativizes, so $Y_0$ is elementary for
$\exists Y\,\eta$ as well. Write
$\eta^{Y_0}(m,w,z)$ for the relativization of $\eta$ evaluated at the
sections $(Y_0)_w,(Y_0)_z$; this formula is arithmetical in $Y_0$.

We now pick out a choice sequence for $\eta$ from among the sections
of $Y_0$. Define pairs $\langle w_m,z_m\rangle$ by recursion on $m$,
with $w_m$ the index of the history accumulated before step $m$ and
$z_m$ the index of the column chosen at step $m$:
\begin{align*}
w_0 &:= \text{the least $w$ such that } (Y_0)_w=\emptyset;\\
z_m &:= \text{the least $z$ such that $\eta^{Y_0}(m,w_m,z)$,
        if such a $z$ exists,}\\
        &\hphantom{{}:={}}\text{and otherwise the least $z$ such that
        $(Y_0)_z=\emptyset$;}\\
w_{m+1} &:= \text{the least $w$ such that $(Y_0)_w$ is $(Y_0)_{w_m}$
        extended by $(Y_0)_{z_m}$ as column $m$.}
\end{align*}
The set required of $w_{m+1}$ is $\Delta^0_1$-definable from the
sections $(Y_0)_{w_m}$ and $(Y_0)_{z_m}$ and is therefore a section;
since $\emptyset$ is a section too, the minimizations are nonvacuous.

The pair $\langle w_{m+1},z_{m+1}\rangle$ is produced from
$\langle w_m,z_m\rangle$ by a rule that is arithmetical in $Y_0$, and
a function defined by primitive recursion from an arithmetical rule
is itself arithmetical: its graph is expressed by ``there is a coded
finite sequence of pairs that starts with $\langle w_0,z_0\rangle$,
obeys the rule at every step, and ends with $\langle
w_m,z_m\rangle$''. Hence the graph of $m\mapsto\langle
w_m,z_m\rangle$ is arithmetical in $Y_0$.

By arithmetical comprehension we may therefore form the set
\[
Z:=\{\langle m,k\rangle: k\in (Y_0)_{z_m}\},
\]
that is, the set with $(Z)_m=(Y_0)_{z_m}$ for all $m$; induction on
$m$ gives $(Z)_{<m}=(Y_0)_{w_m}$. We claim
that $Z$ is as required by strong $\DC{n}$. Suppose
$\exists Y\,\eta(m,(Z)_{<m},Y)$ holds. Since $(Z)_{<m}$ is the
section $(Y_0)_{w_m}$, elementarity for $\exists Y\,\eta$ gives
$\exists z\,\eta^{Y_0}(m,w_m,z)$. So the first case in the definition
of $z_m$ applied, and $\eta^{Y_0}(m,w_m,z_m)$ holds.
Elementarity for $\eta$ converts this into
$\eta(m,(Z)_{<m},(Z)_m)$, as required. Finally, $\eta$ and $\chi$ are
$\Sigma^1_n$, so the instance of reflection used lies in $\Rfn{n}$.

\emph{Dependent choice $\Rightarrow$ reflection.} Fix a target
$\Sigma^1_n$ or $\Pi^1_n$ formula $\varphi^*$ and parameters
$P_0,\dots,P_{k-1}$ (note that without loss of generality we may assume that $k\ge 1$).
We build a single $Z$ with each $P_l\dot\in Z$ that
is elementary for $\varphi^*$.

In the metatheory $\varphi^*$ has finitely many subformulas. Let
$\varphi_1,\dots,\varphi_s$ enumerate all $\varphi$ such that
$\exists Y\,\varphi$ is a subformula of $\varphi^*$, together with all
$\neg\varphi$ such that $\forall Y\,\varphi$ is a subformula of
$\varphi^*$. We note
that up to logical equivalence each $\varphi_i$ is a $\Sigma^1_n$
formula, and we tacitly work with these $\Sigma^1_n$ forms.

Reason in $\ACA+{}$strong $\DC{n}$. Fix an arithmetical enumeration
$m\mapsto(i_m,\vec a_m,\vec C_m)$, defined for $m\ge k$, of all possible
\emph{substitutional variants}
$\exists Y\,\varphi_{i_m}(Y,\vec a_m,\vec C_m)$: here $i_m\le s$, the
entries of $\vec a_m$ are numbers, and $\vec C_m$ is a tuple of set
constants drawn from $C_0,C_1,\dots$. The
enumeration is subject to the condition that $C_j$ occurs in
$\vec C_m$ only for $j< m+k$. 

Let $\eta(m,X,Y)$ say: either $m<k$ and $Y=P_{m}$, or else
$m\ge k$ and $\varphi_{i_{m-k}}(Y,\vec a_{m-k},\vec C_{m-k}^X)$, where
$\vec C_{m-k}^X$ is the result of replacement of each $C_j$ with $(X)_j$.
The resulting formula $\eta$ is clearly equivalent to a $\Sigma^1_n$ formula.
Apply strong $\DC{n}$ to
$\eta$, obtaining $Z$ with, for every $m$,
\[
\exists Y\,\eta\bigl(m,(Z)_{<m},Y\bigr)\ \longrightarrow\
\eta\bigl(m,(Z)_{<m},(Z)_m\bigr).
\]
That is, we get $Z$ such that $(Z)_m=P_m$ for $m<k$ and for $m\ge k$ we get that $(Z)_m$ witnesses  $\exists Y\; \varphi_{i_{m-k}}(Y,\vec a_{m-k},\vec C_{m-k}^Z)$ if it is true.
Thus, due to our choice of the enumeration, for any substitution in any of the $\varphi_i$
of (internal) naturals for the first-order variables and of $(Z)_m$'s for the second-order variables,
if the substituted formula is true, then there is a witness for it among the $(Z)_m$'s.

\emph{Elementarity.} By external induction on the subformulas $\chi$ of
$\varphi^*$ we show $\chi(\vec a,\vec S)\leftrightarrow
\chi^Z(\vec a,\vec S)$ for all numbers $\vec a$ and all vectors $\vec S$ of sections of $Z$. Atomic and arithmetical $\chi$ are absolute; Boolean connectives
and number quantifiers commute with relativization, by the induction
hypothesis. Let $\chi=\exists Y\,\varphi_i$ and suppose
$\chi(\vec a,\vec S)$ with $\vec S$ sections. By the property of $Z$ proven above,
there is a section $D$ of $Z$ that witnesses
$\exists Y\,\varphi_i(Y,\vec a,\vec S)$ if it is true. Thus we have the desired equivalence.
The case of a universally quantified formula is completely analogous.
\end{proof}

\section{Reflection over $\WKLs+\neg\ISig$}\label{sec:wkl}
The isomorphism theorem of Fiori-Carones, Ko{\l}odziejczyk,
Wong and Yokoyama \cite[Theorem~2.1]{FKWY} states that any two coded
$\omega$-models of $\WKLs+\neg\ISig$ over the same countable
first-order universe are isomorphic. Its model-completeness corollary
\cite[Corollary~3.3]{FKWY} is that $\WKLs+\neg\ISig$ proves, for every
$L_2$ formula $\varphi(\vec x,\vec Y)$ without extra free variables, the two equivalences
\begin{align}
\varphi(\vec x,\vec X)\;&\leftrightarrow\;\exists Y \bigl( \vec{X}\dot\in Y\land (\WKLs\land\neg\ISig)^Y\land \varphi^Y(\vec{x},\vec{X})\bigr),
\tag{$\ast_\exists$}\\
\varphi(\vec x,\vec X)\;&\leftrightarrow\;\forall Y \bigl( (\vec{X}\dot\in Y\land (\WKLs\land\neg\ISig)^Y)\to \varphi^Y(\vec{x},\vec{X})\bigr).
\tag{$\ast_\forall$}
\end{align}

It is easy to see that together the two equivalences
in particular entail the L\'evy--Montague reflection
principle:

\begin{theorem}\label{thm:wkllr}
$\WKLs+\neg\ISig$ proves the reflection scheme $\Rfns$.
\end{theorem}

\begin{proof}
  We reason in $\WKLs+\neg\ISig$. We are given $\varphi(\vec{P},\vec{x},\vec{X})$
  and claim that there is $Y$ such that $\vec{P}\dot\in Y$ and $\forall \vec{x},\vec{X} (\varphi(\vec{P},\vec{x},\vec{X})\mathrel{\leftrightarrow} \varphi^Y(\vec{P},\vec{x},\vec{X}))$.

  By $\ast_\exists$ there is $Y$ that contains $\vec{P}$ among its sections and such that $(\WKLs\land\neg\ISig)^Y$. By $\ast_\forall$ applied to $\varphi$, we have that $\varphi(\vec{P},\vec{x},\vec{X})\to \varphi^Y(\vec{P},\vec{x},\vec{X})$ for all $\vec{x}$ and $\vec{X}\dot\in Y$. At the same time by the same principle applied to $\lnot\varphi$, we have $\lnot \varphi(\vec{P},\vec{x},\vec{X})\to \lnot \varphi^Y(\vec{P},\vec{x},\vec{X})$ for all $\vec{x}$ and $\vec{X}\dot\in Y$. Thus we have the desired instance of $\Rfns$.
\end{proof}

\begin{corollary}\label{cor:wkllrcons}
$\RCAs+\neg\ISig$ together with $\Rfns$ is $\Pi^1_1$-conservative over $\RCAs+\neg\ISig$.
\end{corollary}

\begin{proof}
By Theorem~\ref{thm:wkllr} and Proposition~\ref{prop:rfncoll} the
theory is contained in $\WKLs+\neg\ISig$,
which is $\Pi^1_1$-conservative over $\RCAs+\neg\ISig$ by the
Simpson--Smith theorem \cite{SimpsonSmith} relativized as in
\cite[\S3]{FKWY}. 
\end{proof}

\begin{remark}\label{rem:wkl}
$\WKLs+\neg\ISig$ has no $\omega$-models, and its characteristic
theorems (the reflection instances, full collection) fail in
$\omega$-models as simple as $\mathrm{REC}$. So nothing here bears on the
$\omega$-model and conservativity questions over $\WKL_0$ taken up next.
The two provability phenomena live on disjoint classes of models.
They do, however, combine to cover the weak bases themselves, by
splitting on $\ISig$; see Corollary~\ref{cor:weakcons}.
\end{remark}

\section{The $\omega_1$-tower of forcing extensions}\label{sec:towers}

We now treat the central regime of $\RCA$ and $\WKL_0$, where
reflection is unprovable (Corollary~\ref{cor:unif}) but, as we show,
$\Pi^1_1$-conservative. The strategy is this. We start from a
countable model and keep its first-order part $M$ fixed. By a
recursion of length $\omega_1$ we expand the second-order part. At
each successor step we adjoin a \emph{universal set} for the level
built so far, that is, a single set whose sections list that level
exactly; and we close the level under $\Delta^0_1$ comprehension and 
appending of generic paths through unbounded binary trees. At limits we take unions. Reflection is
then proved as in the L\'evy--Montague reflection theorem of set
theory, with the levels of the tower in place of the $V_\alpha$'s.
The tower is increasing and continuous, so a L\"owenheim--Skolem
argument produces, for any given formula, levels that are elementary
in the union with respect to it; the universal set sitting on top of
such a level mirrors the whole model inside a single set. The cardinality considerations are as in $\omega_1$-like models
of arithmetic \cite{MillsParis,EnayatMohsenipour}.

\subsection{Covering, continuity, and reflection}

In this subsection we isolate the abstract properties by which a
model of this shape recovers collection and reflection. The
construction of the model itself follows in the next two
subsections.

\begin{definition}\label{def:covering}
Let $(M,\mathcal{S})\models\RCAs$. An \emph{$M$-indexed family} in
$\mathcal{S}$ is a family $\{Z_y : y\in M\}\subseteq\mathcal{S}$,
indexed in the metatheory by the elements of $M$. For
$X\in\mathcal{S}$ write $\mathcal{S}_X:=\{(X)_y : y\in M\}$ for the
family of sections of $X$, itself an $M$-indexed family in
$\mathcal{S}$ by $\Delta^0_1$ comprehension.
\begin{enumerate}
\item[(i)] $(M,\mathcal{S})$ has the \emph{covering property} if
every $M$-indexed family in $\mathcal{S}$ is contained in
$\mathcal{S}_X$ for a single $X\in\mathcal{S}$.
\item[(ii)] $(M,\mathcal{S})$ has the \emph{continuous covering
property} if the covers can moreover be drawn from a single family
$\mathcal{I}\subseteq\mathcal{S}$ that is closed under unions of
small upwards directed subfamilies. That is: every $M$-indexed family in $\mathcal{S}$ is
contained in $\mathcal{S}_X$ for some $X\in\mathcal{I}$; and
whenever $C\subseteq\mathcal{I}$ has cardinality at most $|M|$ and
$\{\mathcal{S}_X : X\in C\}$ is upwards directed with respect to
inclusion ordering, then there is
$X'\in\mathcal{I}$ with
$\mathcal{S}_{X'}=\bigcup_{X\in C}\mathcal{S}_X$.
\end{enumerate}
\end{definition}

$M$-indexation generalizes countability: a nonempty family is
$M$-indexed exactly when it has at most $|M|$ members, and models of
$\RCAs$ are infinite, so finite and countable subfamilies are always
$M$-indexed. The covering property will give collection. The
continuous covering property will give reflection: the section
families $\mathcal{S}_X$, $X\in\mathcal{I}$, play the role of the
partially built submodels in a downward L\"owenheim--Skolem
argument, covering provides the next submodel, and continuity
provides the limit.

\begin{lemma}\label{lem:covering}
Let $(M,\mathcal{S})\models\RCAs$.
\begin{enumerate}
\item[(i)] If $(M,\mathcal{S})$ has the covering property, then it
satisfies \emph{collection} for all $L_2$ formulas
$\varphi(x,Y)$ with parameters:
$\forall x\,\exists Y\,\varphi(x,Y)\to\exists X\,\forall x\,\exists
y\,\varphi(x,(X)_y)$.
\item[(ii)] If $(M,\mathcal{S})$ has the continuous covering
property, then it satisfies the full reflection scheme $\Rfns$.
\end{enumerate}
\end{lemma}

\begin{proof}
(i) Assume $(M,\mathcal{S})\models\forall x\,\exists Y\,\varphi(x,Y)$.
For each $a\in M$ choose, in the metatheory, $Y_a\in\mathcal{S}$
with $(M,\mathcal{S})\models\varphi(a,Y_a)$; one exists by the
assumption. This is an $M$-indexed family; let $X\in\mathcal{S}$
cover it. For every $a\in M$ we have $\varphi(a,Y_a)$ and
$Y_a=(X)_y$ for some $y$, whence $\exists y\,\varphi(a,(X)_y)$; thus
$\forall x\,\exists y\,\varphi(x,(X)_y)$. Note that $\varphi$ is
evaluated in the single model throughout.

(ii) Reflection is a scheme, so fix an instance $\Rfns_\varphi$ and
parameters $\vec P\in\mathcal{S}$, and let $\Phi$ consist of the
subformulas of $\varphi$, with $\forall Z$ read as
$\neg\exists Z\,\neg$ throughout --- a finite, subformula-closed
family. We run the downward L\"owenheim--Skolem argument for the
formulas of $\Phi$, using the section families $\mathcal{S}_X$,
$X\in\mathcal{I}$, as the partially built submodels. Let
$X_0\in\mathcal{I}$ cover the finite family $\{\vec P\}$. Suppose
$X_n\in\mathcal{I}$ is built. Consider the formulas
$\exists Z\,\psi$ in $\Phi$ and all assignments of elements of $M$ to
the free number variables and of members of $\mathcal{S}_{X_n}$ to
the free set variables other than $Z$. For each, choose in the
metatheory a witness $W\in\mathcal{S}$ with
$(M,\mathcal{S})\models\psi(W,\dots)$, if one exists. There are only
$|M|$ many pairs of a formula and an assignment. So the witnesses,
together with $\mathcal{S}_{X_n}$, form an $M$-indexed family; let
$X_{n+1}\in\mathcal{I}$ cover it. The resulting chain
$\mathcal{S}_{X_0}\subseteq\mathcal{S}_{X_1}\subseteq\cdots$ is
countable, so there is $X\in\mathcal{I}$ with
$\mathcal{S}_X=\bigcup_n\mathcal{S}_{X_n}$.

We claim that $(M,\mathcal{S}_X)$ agrees with $(M,\mathcal{S})$ on
every formula of $\Phi$, at all parameters from $M$ and
$\mathcal{S}_X$. Induct on the members of $\Phi$, the subformula
closure keeping the induction inside $\Phi$. The two structures
share the first-order part, so atomic formulas get the same value,
and negation, conjunction and number quantifiers pass through the
induction. Consider $\exists Z\,\psi$ in $\Phi$. If
$(M,\mathcal{S}_X)\models\exists Z\,\psi$, its witness also belongs
to $\mathcal{S}$, and the agreement on $\psi$ carries the truth up.
If $(M,\mathcal{S})\models\exists Z\,\psi$, the finitely many set
parameters lie in some $\mathcal{S}_{X_n}$; a chosen witness $W$
then lies in $\mathcal{S}_{X_{n+1}}\subseteq\mathcal{S}_X$, and the
agreement on $\psi$ carries the truth down.

Finally, $\vec P\dot\in X$, and the $\omega$-model coded by $X$
(Definition~\ref{def:cl}) is exactly $(M,\mathcal{S}_X)$. So for all
$\vec x\in M$ and all $\vec X\dot\in X$, the relativization
$\varphi^X(\vec P,\vec x,\vec X)$ expresses the satisfaction of
$\varphi(\vec P,\vec x,\vec X)$ in $(M,\mathcal{S}_X)$, and this
agrees with its truth in $(M,\mathcal{S})$. Thus $X$ witnesses
$\Rfns_\varphi$ at $\vec P$.
\end{proof}

The proof of (ii) --- the choice of witnesses and the induction ---
involves only the finitely many subformulas of the instance being
proved. This locality costs nothing here (running the same chain
over all formulas at once, of which there are countably many, would
produce a single $X$ elementary for the full language), but it is
what lets the metamathematical analysis of
Section~\ref{sec:onecon} formalize the argument as it stands, one
instance at a time.

\subsection{Coding by generic listing}

This subsection constructs the universal set of the successor step
of the tower: a single set whose columns list the second-order part
of a given countable model exactly. If a countable family
$\{Y_k : k\in\omega\}$ of nonempty sets is coded
naively over a nonstandard $M$, by putting $Y_k$ in column $k$, then the
nonempty columns of the resulting $X$ form exactly the standard cut,
which is $\Sigma_1(X)$-definable, so $\ISig$ with parameter $X$
fails. Genericity is exactly what prevents the listing from leaving such
a definable trace. (For a standard $M$ the genericity is harmless, and
the lemma applies to every countable model, standard or not.)

\begin{lemma}[Coding lemma]\label{lem:codingns}
Let $(M,\mathcal{S})\models\RCA$ be countable. Then there are a family
$\overline{\mathcal{S}}\supseteq\mathcal{S}$ with
$(M,\overline{\mathcal{S}})\models\RCA$ and a set
$X\in\overline{\mathcal{S}}$ whose $M$-sections are precisely the members
of $\mathcal{S}$, such that every member of $\overline{\mathcal{S}}$ is
$\Delta_1(X)$ in the sense of $M$.
\end{lemma}

The proof occupies the rest of the subsection. Following the
proof-theoretic use of forcing surveyed by Avigad \cite{Avigad2004},
we set the forcing up internally, as a scheme over $\RCAs$:
Definition~\ref{def:collapse} introduces the forcing relation,
Lemma~\ref{lem:logic} establishes its logic, Lemma~\ref{lem:impbq}
derives the clause for bounded quantification,
Lemma~\ref{lem:complexity} computes the arithmetic complexity of
forcing, and
Lemmas~\ref{lem:transfer} and~\ref{lem:forceind} turn induction in
the ground model into forced induction. Models and generic filters
enter only at the end (Definition~\ref{def:generic}), through the
truth lemma (Lemma~\ref{lem:truth}) and the assembly of the
extension.

\begin{definition}[$\RCAs$]\label{def:collapse}
The \emph{listing forcing} is defined as follows. A
\emph{condition} is a pair $p=(W_p,a_p)$ of a
set and a number, read as the finite list
$\langle (W_p)_i:i<a_p\rangle$ of the first $a_p$ columns of $W_p$;
columns beyond $a_p$ are immaterial. Put $q\le p$ iff $a_q\ge a_p$
and $(W_q)_i=(W_p)_i$ for $i<a_p$, i.e.\ $q$ end-extends the list of
$p$. The \emph{forcing language} is the first-order language of
arithmetic with one set constant $X$ and connectives
$\neg,\wedge,\to,\forall$ ($\vee,\exists$ abbreviated as usual). For
a sentence $\varphi$ of the forcing language --- numerical
parameters allowed, here and throughout --- and a condition $p$,
define $p\Vdash\varphi$ by recursion on $\varphi$:
\begin{itemize}\setlength\itemsep{1pt}
\item $X$-free atomic $\alpha$: $p\Vdash\alpha$ iff $\alpha$ holds;
\item $p\Vdash(k\in X)$ iff $k=\langle i,m\rangle$ with $i<a_p$
and $k\in W_p$;
\item $p\Vdash\varphi\wedge\psi$ iff $p\Vdash\varphi$ and $p\Vdash\psi$;
\item $p\Vdash\neg\varphi$ iff no $q\le p$ has $q\Vdash\varphi$;
\item $p\Vdash\varphi\to\psi$ iff every $q\le p$ with
$q\Vdash\varphi$ has $q\Vdash\psi$;
\item $p\Vdash\forall x\,\varphi(x)$ iff $p\Vdash\varphi(n)$ for all
$n$.
\end{itemize}
\end{definition}

For fixed $\varphi$ the clauses unwind to an $L_2$ formula in $p$
and the parameters, set quantifiers entering through the clauses
for negation and implication.
All assertions about $\Vdash$ below are therefore schemes, one
statement per formula, each proved in the theory named with the
statement.

\begin{remark}[The Kripke reading]\label{rem:kripke}
Within any particular model of $\RCAs$, the definition carves out a
first-order Kripke model, definable from within that model: the
nodes are the conditions, extension is accessibility, the domain at
every node is the first-order universe --- constant, and with every
element named by a parameter --- and the atomic clauses give the
valuation. The clauses for $\neg,\wedge,\to,\forall$ are verbatim
the Kripke semantics of intuitionistic logic, so the unwound $L_2$
formula $p\Vdash\varphi$ says precisely that $\varphi$ is forced at
the node $p$ of this Kripke model. The scheme $\Vdash$ is thus an
interpretation, in the classical theory $\RCAs$, of the classical
theory of a Kripke model of the forced theory, a Kripke model being
just a particular kind of first-order structure; this is the
treatment of forcing of \cite{PakhomovVisser}. Consequently the
standard facts of Kripke semantics are available as theorem schemes
of $\RCAs$: their textbook proofs are external inductions on a
formula or on a derivation, each instance a finite unwinding of the
clauses, carried out in $\RCAs$ and consuming no induction of the
theory. What the standard facts take as input is two properties of
the atomic valuation, checked here once and for all. \emph{Atoms
are monotone}: an $X$-free atom does not mention the condition, and
an atom $k\in X$ forced by $p$ has $k=\langle i,m\rangle$ with its
column frozen, $i<a_p$, so every $q\le p$, having
$(W_q)_i=(W_p)_i$, forces it as well. \emph{Atoms are stable}: if
$p\nVdash\alpha$, then some $q\le p$ forces $\neg\alpha$ --- $p$
itself when $\alpha$ is $X$-free or its column is frozen in $p$, no
extension then changing the evaluation, and the extension freezing
the column as the empty set otherwise. Classical logic is recovered
on the $\vee,\exists$-free fragment over the stable atoms: reading
$\vee$ and $\exists$ as abbreviations is, up to the double negation
of atoms, the G\"odel--Gentzen translation, built into the syntax.
This settles the logic of forcing, as the next lemma records.
\end{remark}

\begin{lemma}[Logic of forcing; $\RCAs$]\label{lem:logic}
For all formulas $\varphi$ of the forcing language and all
conditions $p$:
\begin{enumerate}
\item[(i)] \emph{Persistence:} if $p\Vdash\varphi$ and $q\le p$, then
$q\Vdash\varphi$.
\item[(ii)] \emph{Consistency:} no condition forces both $\varphi$
and $\neg\varphi$.
\item[(iii)] \emph{Intuitionistic closure:} the sentences forced by
a fixed condition are closed under intuitionistic predicate logic,
in the official connectives $\neg,\wedge,\to,\forall$.
\item[(iv)] \emph{Deductive closure:} the sentences forced by a
fixed condition are closed under classical predicate logic.
\item[(v)] \emph{Stability:} $p\Vdash\varphi$ iff
$p\Vdash\neg\neg\varphi$; equivalently, $p\nVdash\varphi$ iff some
$q\le p$ forces $\neg\varphi$.
\item[(vi)] \emph{Decision:} the class
$D_\varphi:=\{q:q\Vdash\varphi\text{ or }q\Vdash\neg\varphi\}$ is
dense: every condition has an extension in $D_\varphi$.
\end{enumerate}
\end{lemma}

\begin{proof}
(i) The monotonicity lemma of Kripke semantics, an induction on
$\varphi$ needing only the monotone atoms of
Remark~\ref{rem:kripke}: the clauses for $\neg$ and $\to$ are
persistent outright, extensions of $q$ being extensions of $p$, and
$\wedge,\forall$ pass through the induction.

(ii) If $p\Vdash\neg\varphi$, then no $q\le p$ forces $\varphi$; and
$p\le p$.

(iii) The soundness of intuitionistic predicate logic for Kripke
semantics, applied proof by proof: for a fixed derivation of
$\varphi$ from sentences forced by $p$, an external induction on
the derivation --- persistence entering where hypotheses are
discharged --- verifies $p\Vdash\varphi$, in $\RCAs$ as
Remark~\ref{rem:kripke} explains.

(iv) Every instance of the stability scheme $\neg\neg\psi\to\psi$
is forced at every condition. For an atom $\alpha$: a condition
forcing $\neg\neg\alpha$ --- no extension of it forces
$\neg\alpha$ --- forces $\alpha$, by the stability of atoms
(Remark~\ref{rem:kripke}), and the clause for $\to$ concludes. For
compound $\psi$, intuitionistic logic proves $\neg\neg\psi\to\psi$
from the stability of the atoms occurring in it, so the instances
are forced by the atomic case and (iii). And in our
$\vee,\exists$-free language, classical predicate logic is
axiomatized over intuitionistic predicate logic precisely by the
stability scheme; for this and the preceding standard fact see
\cite[Lemma~2.2, Proposition~2.3]{Avigad2004}, or
\cite[Ch.~2, \S3]{TroelstraVanDalen}. Closure under classical
deduction thus reduces, through the forced stability instances, to
closure under intuitionistic deduction, which is (iii).

(v) If $p\Vdash\varphi$, then by persistence and consistency no
$q\le p$ forces $\neg\varphi$, that is, $p\Vdash\neg\neg\varphi$.
Conversely, $\neg\neg\varphi\to\varphi$ is forced --- the stability
scheme of (iv) --- so the clause for $\to$ turns
$p\Vdash\neg\neg\varphi$ into $p\Vdash\varphi$. For the second
formulation, take the contrapositive and unwind the clause for
$\neg$ once: $p\nVdash\varphi$ iff $p\nVdash\neg\neg\varphi$ iff
some $q\le p$ forces $\neg\varphi$.

(vi) Given $p$: if some $q\le p$ forces $\varphi$, then $q\in
D_\varphi$; if none does, then $p\Vdash\neg\varphi$ by the clause for
negation, and $p\in D_\varphi$.
\end{proof}

The bounded number quantifier has the expected derived clause; we
record it for later use.

\begin{lemma}[Bounded quantification; $\RCAs$]
\label{lem:impbq}
For all formulas of the forcing language and all conditions $p$:
$p\Vdash\forall y\!\le\!u\,\varphi(y)$ iff $p\Vdash\varphi(k)$ for
every $k\le u$.
\end{lemma}

\begin{proof}
The bounded quantification abbreviates
$\forall y\,(y\le u\to\varphi(y))$, so $p$ forces it iff
$p\Vdash(k\le u\to\varphi(k))$ for every $k$. The
atom $k\le u$ is $X$-free. If it is true, every condition
forces it, and the clause for $\to$ turns the implication into
``every $q\le p$ forces $\varphi(k)$'', which by persistence is just
$p\Vdash\varphi(k)$. If it is false, no condition forces it,
and the implication is forced vacuously. So exactly the instances
with $k\le u$ remain.
\end{proof}

The forcing relation, on arithmetic formulas of low complexity, is
itself of low complexity. To compute it we substitute coded finite
sets for the set constant. Recall the Ackermann membership relation:
$k\inack s$ iff the $k$th digit of the binary expansion of $s$ is
$1$, that is, iff $\lfloor s/2^k\rfloor$ is odd. With exponentiation
in the signature this is a $\Delta_0$ relation, and bounded
comprehension is available: for every $\Delta^0_0$ formula $\varphi$,
possibly with
parameters, and every $b$ there is $s<2^b$ with
$\forall k\!<\!b\,(k\inack s\leftrightarrow\varphi(k))$. For a
formula $\theta$ of the forcing language and a number $s$, let
$\theta[X{:=}s]$ be the arithmetic formula obtained by replacing
each atom $u\in X$ by $u\inack s$; it is $\Delta_0$ whenever
$\theta\in\Delta_0(X)$.

Next, assign to each $\theta\in\Delta_0(X)$ a \emph{bounding term}
$t_\theta$, in the variables of $\theta$, by recursion on the
construction of $\Delta_0(X)$ formulas:
\begin{gather*}
t_{u\in X}:=u,\qquad
t_\alpha:=0\ \ (\alpha\ X\text{-free atomic}),\qquad
t_{\neg\theta}:=t_\theta,\\
t_{\theta\wedge\psi}:=t_{\theta\to\psi}:=t_\theta+t_\psi,\qquad
t_{\forall y\le u\,\theta}:=t_\theta[y{:=}u].
\end{gather*}
Terms are monotone in each variable, so the evaluation of $\theta$,
under any interpretation of the constant $X$, depends on $X$ only at
numbers up to $t_\theta$
--- an induction on $\theta$, the monotonicity entering at the
bounded quantifier step. Finally, say that $s$ \emph{agrees with $p$
below $t$}, written $s\approx_t p$, if $s$ and the list of $p$ have
the same members among the pairs at most $t$ whose column $p$
freezes:
\[
s\approx_t p\ :\iff\ \forall i\!<\!a_p\,\forall m\!\le\!t\,
\bigl(\langle i,m\rangle\le t\to
(\langle i,m\rangle\inack s\leftrightarrow
\langle i,m\rangle\in W_p)\bigr).
\]
It is easy to see that $s\approx_t p$ can be expressed by a $\Delta^0_0$-formula.
By bounded comprehension every condition
induces canonical agreeing sets: writing $p{\upharpoonright}t$ for
the coded finite set whose members are the pairs
$\langle i,m\rangle\le t$ with $i<a_p$ and
$\langle i,m\rangle\in W_p$, we have $p{\upharpoonright}t<2^{t+1}$
and $p{\upharpoonright}t\approx_t p$. Forcing a $\Delta_0(X)$
formula then means that it holds under every agreeing substitution:

\begin{lemma}[Complexity of forcing; $\RCAs$]\label{lem:complexity}
\begin{enumerate}
\item[(i)] For $\theta\in\Delta_0(X)$,
\begin{align*}
p\Vdash\theta\ &\iff\
\forall s\!<\!2^{t_\theta+1}\,\bigl(s\approx_{t_\theta} p\to\theta[X{:=}s]\bigr)\\
&\iff\
\forall s\,\bigl(s\approx_{t_\theta} p\to\theta[X{:=}s]\bigr).
\end{align*}
Hence $p\Vdash\theta$ is a $\Delta^0_0$ predicate of $p$ and the
parameters of $\theta$ with the set parameter $W_p$.
\item[(ii)] For $\psi=\forall w\,\theta(w)$ with
$\theta\in\Delta_0(X)$, the relation $p\Vdash\psi$ is a
$\Pi^0_1$ predicate of $p$ and the parameters of $\psi$ with the
set parameter $W_p$.
\end{enumerate}
\end{lemma}

\begin{proof}
(i) Write $\bar\theta(p)$ for the unbounded substitution condition,
$\forall s\,\bigl(s\approx_{t_\theta} p\to\theta[X{:=}s]\bigr)$.
It is equivalent to its bounded companion
$\forall s\!<\!2^{t_\theta+1}\,\bigl(s\approx_{t_\theta}
p\to\theta[X{:=}s]\bigr)$, and in both the agreement bound
$t_\theta$ may be raised to any $t\ge t_\theta$. Indeed,
$\theta[X{:=}s]$ depends only on the members of $s$ up to
$t_\theta$. So take any $s\approx_{t_\theta}p$ and replace its
members above $t_\theta$ by those of $p{\upharpoonright}t$: the
result is smaller than $2^{t+1}$, agrees with $p$ below $t$, and
satisfies $\theta[X{:=}\,\cdot\,]$ iff $s$ does. It remains to
prove $p\Vdash\theta\iff\bar\theta(p)$. We
argue by induction on the construction of $\Delta_0(X)$ formulas.

\emph{Atoms.} Suppose first that $\theta$ is an $X$-free atom. The
substitution does not change $\theta$, and at least one agreeing
$s$ exists, namely $p{\upharpoonright}0$; so both sides are
equivalent to $\theta$ itself. Now let $\theta$ be an atom
$k\in X$, here by definition $t_\theta=k$. If $p\Vdash k\in X$, then
$k=\langle i,m\rangle$ with $i<a_p$ and $k\in W_p$, and the
agreement condition at this pair guarantees $k\inack s$ for every
$s\approx_k p$. If $p\nVdash k\in X$, then $p{\upharpoonright}k$ is
an agreeing set that omits $k$, so $\bar\theta(p)$ fails.

\emph{Negations.} Since $t_{\neg\theta}=t_\theta$, the statement
$\overline{\neg\theta}(p)$ says: no $s\approx_{t_\theta}p$
satisfies $\theta[X{:=}s]$.

Suppose $p\Vdash\neg\theta$, and suppose towards a contradiction
that some $s\approx_{t_\theta}p$ satisfies $\theta[X{:=}s]$. Extend
$p$ to a condition $q$ that freezes $s$: put
$a_q:=\max(a_p,t_\theta+1)$, and let $W_q$ keep the columns of
$W_p$ below $a_p$ and read the columns from $a_p$ on off $s$. Every
number up to $t_\theta$ is a pair whose column $q$ freezes, and it
belongs to $W_q$ iff it belongs to $s$: below $a_p$ this is the
agreement, above it is the construction. Hence every
$s'\approx_{t_\theta}q$ has the same members up to $t_\theta$ as
$s$, and therefore satisfies $\theta[X{:=}s']$. This is
$\bar\theta(q)$; by the induction hypothesis $q\Vdash\theta$,
contradicting $p\Vdash\neg\theta$.

Conversely, suppose $\overline{\neg\theta}(p)$, and let $q\le p$.
If $q\Vdash\theta$, then $\theta[X{:=}s]$ holds for
$s=q{\upharpoonright}t_\theta$ by the induction hypothesis; but
this $s$ agrees with $q$ and hence with $p$, contradicting
$\overline{\neg\theta}(p)$. So no extension of $p$ forces $\theta$,
i.e.\ $p\Vdash\neg\theta$.

\emph{Conjunctions and bounded quantifiers.} These cases are
routine, because everything in sight commutes with the connective.
Forcing does: $p\Vdash\theta\wedge\psi$ iff $p\Vdash\theta$ and
$p\Vdash\psi$, and $p\Vdash\forall y\!\le\!u\,\theta$ iff
$p\Vdash\theta(k)$ for all $k\le u$, by Lemma~\ref{lem:impbq}.
The substitution does: it acts on each part separately. And the
bounds match up: $t_{\theta\wedge\psi}$ and
$t_{\forall y\le u\,\theta}$ dominate the bounds of the parts,
terms being monotone, so by the opening remark the statements
$\bar\theta(p)$, $\bar\psi(p)$ and $\overline{\theta(k)}(p)$ may
all be taken at the bound of the compound. The induction hypothesis
for the parts then gives the claim for the compound, after
exchanging two universal quantifiers.

\emph{Implications.} By deductive closure
(Lemma~\ref{lem:logic}(iv)), a condition forces $\theta\to\psi$ iff
it forces $\neg(\theta\wedge\neg\psi)$; the substitution instances
of the two are equivalent outright, and the bounding terms agree,
$t_{\theta\to\psi}=t_{\neg(\theta\wedge\neg\psi)}$. The case thus
reduces to the preceding ones.

(ii) By the clause for $\forall$, $p\Vdash\psi$ iff
$p\Vdash\theta(k)$ for every $k$; and by (i), each statement
$p\Vdash\theta(k)$ is in turn equivalent to its unbounded
substitution condition. Chaining the two equivalences,
\[
p\Vdash\psi\iff
\forall k\,\forall s\,\bigl(s\approx_{t_{\theta(k)}}p\to
\theta(k)[X{:=}s]\bigr).
\]
The right-hand side has the required form: behind the two
universal number quantifiers $\forall k\,\forall s$ stands a
$\Delta^0_0$ matrix, since the value of the bounding term
$t_{\theta(k)}$ is computed from $k$ and the parameters, the
agreement relation $\approx$ is $\Delta^0_0$ with the set
parameter $W_p$, and the substituted formula $\theta(k)[X{:=}s]$
is $\Delta_0$. So $p\Vdash\psi$ is a $\Pi^0_1$ predicate of $p$
and the parameters of $\psi$, with the set parameter $W_p$.
\end{proof}

\begin{lemma}[Induction transfer; $\RCAs$]\label{lem:transfer}
Let $\psi(x)$ be a formula of the forcing language, and suppose
that, for every condition $p$, the formula $p\Vdash\psi(x)$
satisfies induction in the variable $x$. Then every condition
forces the induction axiom
$[\psi(0)\wedge\forall x(\psi(x)\to\psi(x+1))]\to\forall x\,\psi(x)$.
\end{lemma}

\begin{proof}
Suppose a condition $p$ forces the premise: $p\Vdash\psi(0)$ and
$p\Vdash\forall x\,(\psi(x)\to\psi(x+1))$, the latter giving
$p\Vdash(\psi(n)\to\psi(n+1))$ for every $n$. So the formula
$p\Vdash\psi(x)$ holds at $x=0$ and, by the clause for $\to$,
passes from each $x$ to $x+1$; by the induction assumed for it,
it holds at every $x$, that is, $p$ forces the conclusion
$\forall x\,\psi(x)$.
Since every condition forcing the premise thus forces the
conclusion, the clause for $\to$ makes every condition force the
axiom.
\end{proof}

Call a sentence of the forcing language \emph{forced} if every
condition forces it. Write
$\mathsf{EA}(X)$ for elementary arithmetic in $X$: the basic axioms,
those of exponentiation included, and $\Delta_0(X)$ induction.
Write $\mathsf{I}\Sigma_1(X)$ and $\mathsf{I}\Pi_1(X)$ for the
schemes of induction for $\Sigma_1(X)$ and $\Pi_1(X)$ formulas
--- first-order theories in the constant $X$. By
Lemma~\ref{lem:transfer}, an induction scheme is
forced as soon as the base theory proves induction in $x$ for the
corresponding formulas $p\Vdash\psi(x)$; the two bases of interest
line up as follows.

\begin{lemma}[Forced induction]\label{lem:forceind}
\begin{enumerate}
\item[(i)] ($\RCAs$) Every axiom of $\mathsf{EA}(X)$ is forced.
\item[(ii)] ($\RCA$) Every instance of $\mathsf{I}\Pi_1(X)$ is
forced, and with it, by deductive closure, every consequence of
$\mathsf{EA}(X)+\mathsf{I}\Pi_1(X)$ --- in particular every
instance of $\mathsf{I}\Sigma_1(X)$.
\end{enumerate}
\end{lemma}

\begin{proof}
(i) An $X$-free sentence is forced if and only if it holds: by
induction on the formula, forcing of an $X$-free formula does not
depend on the condition, so the quantifiers over extensions in the
clauses for $\neg$ and $\to$ are inert. So the
basic axioms, those of exponentiation included, are forced. For an
instance of $\Delta_0(X)$ induction, with induction formula
$\psi(x)\in\Delta_0(X)$, the formula $p\Vdash\psi(x)$ is
$\Delta^0_0$ with the parameters $p$ and $W_p$ by
Lemma~\ref{lem:complexity}(i), and $\RCAs$ has $\Delta^0_0$
induction with parameters; Lemma~\ref{lem:transfer} applies.

(ii) For $\psi(x)\in\Pi_1(X)$ the formula $p\Vdash\psi(x)$ is
$\Pi^0_1$ with the parameters $p$ and $W_p$ by
Lemma~\ref{lem:complexity}(ii), and $\RCA$ has $\Pi^0_1$ induction
with set parameters, the $\Sigma_1$ and $\Pi_1$ induction schemes
being equivalent over the basic axioms
\cite[Ch.~I, \S2(b)]{HajekPudlak}; Lemma~\ref{lem:transfer}
applies. The sentences forced by a fixed condition are deductively
closed (Lemma~\ref{lem:logic}(iv)) and contain, by (i) and the
above, $\mathsf{EA}(X)+\mathsf{I}\Pi_1(X)$, hence
$\mathsf{I}\Sigma_1(X)$ by the same equivalence.
\end{proof}

With the internal development complete, we pass to models and generic
filters.

\begin{definition}\label{def:generic}
Let $(M,\mathcal{S})\models\RCAs$ be countable, and interpret the
forcing of Definition~\ref{def:collapse} in $(M,\mathcal{S})$: the
conditions form a partial order $\mathbb{P}$ definable over
$(M,\mathcal{S})$, and $p\Vdash\varphi$ is the unwound $L_2$
formula evaluated there. A filter
$G\subseteq\mathbb{P}$ (upward closed and downward directed) is
\emph{generic} if it meets every dense subclass of $\mathbb{P}$ that
is definable over $(M,\mathcal{S})$ with parameters from $M$ and
$\mathcal{S}$. From a filter $G$ read off the generic object
\[
X_G:=\{\langle i,m\rangle:\exists p\in G\,(i<a_p\wedge\langle
i,m\rangle\in W_p)\},
\]
and write $(M,G)$ for $(M,X_G)$.
\end{definition}

Since $M$ and $\mathcal{S}$ are countable, there are only countably
many dense definable subclasses, so a generic filter through any
prescribed condition exists, obtained by meeting them one after
another. Forcing accounts for all truth in the generic extensions:

\begin{lemma}[Truth lemma]\label{lem:truth}
In the setting of Definition~\ref{def:generic}, for every sentence
$\varphi$ of the forcing language and every condition $p$,
\[
p\Vdash\varphi\quad\Longleftrightarrow\quad
(M,G)\models\varphi\ \text{for every generic }G\ni p;
\tag{$\dagger$}
\]
and for every generic filter $G$,
\[
(M,G)\models\varphi\quad\Longleftrightarrow\quad
\exists p\in G\ p\Vdash\varphi.
\tag{$\ddagger$}
\]
\end{lemma}

\begin{proof}
Replacing each subformula $\chi\to\psi$ by
$\neg(\chi\wedge\neg\psi)$ is a classical equivalence, so it
changes neither satisfaction nor, by deductive closure
(Lemma~\ref{lem:logic}(iv)), forcing; we may therefore assume the
connectives are $\neg,\wedge,\forall$.
Since $(\Leftarrow)$ of $(\ddagger)$ is contained in $(\dagger)$, we
prove $(\dagger)$ and the forward direction of $(\ddagger)$ together,
by a simultaneous induction on $\varphi$.

\emph{Atoms and conjunctions.} The atomic cases are immediate from
the definition of $X_G$, and the $\wedge$ cases too, downward
directedness giving $(\ddagger)$.

\emph{Universal quantifiers.} For $(\dagger)$, two universal
quantifiers commute: $p\Vdash\forall x\,\varphi$ iff
$\forall n\,p\Vdash\varphi(n)$ iff, by the induction hypothesis,
$\forall n\,\forall G\ni p\,(M,G)\models\varphi(n)$ iff
$\forall G\ni p\,(M,G)\models\forall x\,\varphi$. For
$(\ddagger)$, the class of conditions that
force $\forall x\,\varphi$ or force $\neg\varphi(n)$ for some $n$ is
definable, and it is dense: a condition with no extension forcing any
$\neg\varphi(n)$ forces every $\varphi(n)$ by stability
(Lemma~\ref{lem:logic}(v)), hence forces $\forall x\,\varphi$. A
generic $G$ with $(M,G)\models\forall x\,\varphi$ meets the class at
some $r$. And $r\Vdash\neg\varphi(n)$ is impossible: by $(\ddagger)$
for $\varphi(n)$ some $q\in G$ forces $\varphi(n)$, and a common
extension of $r$ and $q$ in $G$ would be an extension of $r$ forcing
$\varphi(n)$, which $r\Vdash\neg\varphi(n)$ forbids. So
$r\Vdash\forall x\,\varphi$.

\emph{Negations.} For $(\dagger)$, suppose first that
$p\nVdash\neg\varphi$; then
some $q\le p$ forces $\varphi$, so by the induction hypothesis every
generic $G\ni q$ --- one exists, and it contains $p$ --- has
$(M,G)\models\varphi$, and the right-hand side fails. Suppose
conversely that $p\Vdash\neg\varphi$ while some generic $G\ni p$ had
$(M,G)\models\varphi$. By $(\ddagger)$ for $\varphi$, some $r\in G$
forces $\varphi$, and a common extension of $p$ and $r$ in $G$ would
be a condition below $p$ forcing $\varphi$ --- against
$p\Vdash\neg\varphi$. For $(\ddagger)$, a generic $G$ with
$(M,G)\models\neg\varphi$ meets the decision class $D_\varphi$, dense
by Lemma~\ref{lem:logic}(vi), at some $r$. This $r$ cannot force
$\varphi$, since then $(M,G)\models\varphi$ by $(\dagger)$. So
$r\Vdash\neg\varphi$.
\end{proof}

\begin{definition}\label{def:rec}
For a model $M$ of the language of $\mathsf{EA}(X)$, let
$\mathsf{Rec}(M)$ be the $L_2$ model whose first-order part is $M$
with the constant $X$ forgotten, and whose second-order part is the
collection of all sets $\Delta_1(X)$-definable in $M$, number
parameters allowed.
\end{definition}

The following is a special case of \cite[Lemma~IX.1.8]{Simpson}:

\begin{lemma}\label{lem:sigmadelta}
For every model $M$ of $\mathsf{EA}(X)+\mathsf{I}\Sigma_1(X)$, we
have $\mathsf{Rec}(M)\models\RCA$.
\end{lemma}

\begin{proof}[Proof of Lemma~\ref{lem:codingns}]
\emph{Exactness.} Fix a generic filter $G$ for the forcing
interpreted in $(M,\mathcal{S})$, and set $X:=X_G$. For $b\in M$
the class $\{p:a_p\ge b\}$ is dense (pad with empty columns), and for
$Y\in\mathcal{S}$ the class $\{p:Y=(W_p)_i\text{ for some }i<a_p\}$
is dense (append $Y$); meeting the former gives $X$ its $M$-many
columns, and meeting the latter puts every $Y\in\mathcal{S}$ among
them. Conversely, a column $(X)_i$ is frozen once a condition with
$a_p>i$ enters $G$, whence $(X)_i\in\mathcal{S}$. So the columns of
$X$ are precisely the members of $\mathcal{S}$.

\emph{Assembly.} The ground model $(M,\mathcal{S})$ satisfies
$\RCA$, so Lemma~\ref{lem:forceind}(i),(ii) holds in it: every
axiom of $\mathsf{EA}(X)$ and every instance of
$\mathsf{I}\Sigma_1(X)$ is forced in the sense of
$(M,\mathcal{S})$, and hence true in $(M,X)=(M,G)$ by
$(\ddagger)$. Thus $(M,X)\models\mathsf{EA}(X)+\mathsf{I}\Sigma_1(X)$,
and Lemma~\ref{lem:sigmadelta} gives
$\mathsf{Rec}(M,X)\models\RCA$. Let $\overline{\mathcal{S}}$ be the
second-order part of $\mathsf{Rec}(M,X)$: every member of
$\overline{\mathcal{S}}$ is $\Delta_1(X)$ in the sense of $M$
by construction, and
$\overline{\mathcal{S}}\supseteq\mathcal{S}\cup\{X\}$, the members
of $\mathcal{S}$ being columns of $X$.
\end{proof}

\begin{remark}\label{rem:fullind}
The construction preserves more induction than the lemma records. If
$(M,\mathcal{S})$ satisfies the full induction scheme --- induction
for all $L_2$ formulas --- then for fixed $\psi$ the relation
$p\Vdash\psi(x)$ is an $L_2$ formula, set
quantifiers entering through the clauses for negation and
implication, and full
induction in the ground model makes Lemma~\ref{lem:transfer}
applicable to every formula of the forcing language. So the whole
induction scheme in the constant $X$ is forced, hence true in
$(M,X)$; and quantifiers over $\overline{\mathcal{S}}$ being
replaceable by number quantifiers over $\Delta_1(X)$ definitions,
the full induction scheme holds in $(M,\overline{\mathcal{S}})$ as
well. One intermediate case is available as well, specific to
$\Pi_2$: over $\WKLs$, $\Pi^0_2$
induction implies that $\Pi_2(X)$ induction is forced. The point is
the complexity of forcing one level above
Lemma~\ref{lem:complexity}: the forcing
$p\Vdash\exists x\,\theta(x)$ of a $\Sigma_1(X)$ formula is
$\Pi^0_2$-expressible --- though this does require some work to
verify --- whence $p\Vdash\psi(x)$ is $\Pi^0_2$ for
$\psi(x)\in\Pi_2(X)$, and Lemma~\ref{lem:transfer} applies. The
tower construction below needs none of this, so we record
it only as a remark.
\end{remark}

The base $\RCA$ of Lemma~\ref{lem:codingns} enters at a single
point: the assembly needs $\mathsf{I}\Sigma_1(X)$ to be forced, and
obtains it, through deductive closure, from the forced
$\mathsf{I}\Pi_1(X)$ of Lemma~\ref{lem:forceind}(ii), whose proof
runs the induction transfer on a $\Pi^0_1$ forcing predicate and so
uses $\Sigma^0_1$ induction in the ground model. Over the weak base
$\RCAs$ the natural substitute is the collection scheme
$\mathrm{B}\Sigma_1(X)$. A positive answer to the following problem
would yield an analogue of Lemma~\ref{lem:codingns} with $\RCAs$ in
place of $\RCA$, and with it the prospect of running the tower
construction of the next subsection over the weak bases directly.

\begin{problem}\label{prob:forcebsig}
Does $\RCAs$ prove that every instance of the collection scheme
$\mathrm{B}\Sigma_1(X)$,
\[
\forall x\!<\!a\,\exists w\,\theta(x,w)\ \longrightarrow\
\exists b\,\forall x\!<\!a\,\exists w\!<\!b\,\theta(x,w),
\qquad\theta\in\Delta_0(X),
\]
is forced in the listing forcing of Definition~\ref{def:collapse}?
Does $\WKLs$?
\end{problem}

\subsection{Transfinite towers of extensions and reflection}

In this subsection we put the pieces together. Iterating the coding
lemma along a chain of length $\omega_1$ produces a model with the
continuous covering property, which Lemma~\ref{lem:covering}
converts into reflection; the conservation theorems then follow by
starting the tower from a countable model of the negated sentence.

The following lemma is a standard result, proved by Harrington by
forcing with infinite subtrees \cite[Theorem~IX.2.1]{Simpson};
alternatively, it can be obtained from the H\'ajek--Ku\v{c}era
formalization of the low basis theorem in $\mathsf{I}\Sigma_1(X)$
\cite{HajekKucera,Hajek1993}.

\begin{lemma}\label{lem:wklext}
Every countable model $(M,\mathcal{S})$ of $\RCA$ extends, over the
same first-order universe, to a countable model $(M,\mathcal{S}')$
of $\WKL_0$, $\mathcal{S}'\supseteq\mathcal{S}$.
\end{lemma}

\begin{theorem}\label{thm:towers}
Every countable model $(M,\mathcal{S}_0)$ of $\RCA$ extends, over
the same first-order universe, to a model $(M,\mathcal{S})$ of
$\WKL_0$ with the continuous covering property.
\end{theorem}

\begin{proof}
The chain is defined by recursion on $\alpha<\omega_1$:
\begin{itemize}
\item $\mathcal{S}_0$ is the given family;
\item at stage $\alpha+1$, code all of $\mathcal{S}_\alpha$ into the
sections of a single $X_\alpha$ by Lemma~\ref{lem:codingns}, and
extend the family $\overline{\mathcal{S}_\alpha}\ni X_\alpha$
produced there to a countable $\mathcal{S}_{\alpha+1}$ with
$(M,\mathcal{S}_{\alpha+1})\models\WKL_0$ by
Lemma~\ref{lem:wklext};
\item at limit $\lambda$, take
$\mathcal{S}_\lambda:=\bigcup_{\alpha<\lambda}\mathcal{S}_\alpha$.
\end{itemize}

Let $\mathcal{S}:=\bigcup_{\alpha<\omega_1}\mathcal{S}_\alpha$. Then
$(M,\mathcal{S})\models\WKL_0$, as the union of an increasing chain
of models of $\WKL_0$ with the fixed first-order part $M$; the limit
stages are models of $\WKL_0$ for the same reason, so the recursion
goes through. It remains to verify the continuous covering property,
with $\mathcal{I}:=\{X_\alpha : \alpha<\omega_1\}$ as the family of
covers. First, every $M$-indexed family in $\mathcal{S}$ is
contained in the section family of some member of $\mathcal{I}$:
since $M$ is countable, an $M$-indexed family is just a countable
subfamily $(Y_i : i<\omega)$; taking $\beta_i$ to be the least level
with $Y_i\in\mathcal{S}_{\beta_i}$, the family lies in the stage
$\mathcal{S}_\alpha$ for the countable ordinal
$\alpha:=\sup_i\beta_i$, and $\mathcal{S}_\alpha$ is
\emph{exactly} the section family of $X_\alpha$ by
Lemma~\ref{lem:codingns} --- a property of $(M,X_\alpha)$ alone,
untouched by the later growth of the family. Second, the section
families of the members of $\mathcal{I}$, that is, the stages, are
closed under unions of countable upwards directed families: the
stages are linearly ordered by inclusion, so a directed family of
stages is a chain, and by continuity of the tower the union of a
chain of stages is the stage at the supremum of its indices.
\end{proof}

The listings are exact: the sections of $X_\alpha$ are precisely the
stage $\mathcal{S}_\alpha$. This is what made the covering
continuous, and so, by Lemma~\ref{lem:covering}(ii), it also yields
reflection.

\begin{corollary}\label{cor:rfnext}
Every countable model of $\RCA$ extends, over the same first-order
universe, to a model of $\WKL_0+\Rfns$.
\end{corollary}

\begin{proof}
Extend it by Theorem~\ref{thm:towers};
Lemma~\ref{lem:covering}(ii) turns the continuous covering property
into $\Rfns$.
\end{proof}

\begin{theorem}\label{thm:lr}
$\WKL_0+\Rfns$ is $\Pi^1_1$-conservative over $\RCA$; hence its
first-order part is exactly the set of
$\mathrm{I}\Sigma_1$-provable sentences.
\end{theorem}

\begin{proof}
If $\RCA\nvdash\psi$ for a $\Pi^1_1$ sentence
$\psi=\forall X\,\theta(X)$, fix a countable
$(M,\mathcal{S}_0)\models\RCA+\neg\psi$ with $Z\in\mathcal{S}_0$
witnessing $\neg\theta(Z)$, and extend it to a model of
$\WKL_0+\Rfns$ by Corollary~\ref{cor:rfnext}. The failure
$\neg\theta(Z)$ persists, since its truth depends only on $(M,Z)$.
For the first-order part: a first-order theorem of $\WKL_0+\Rfns$ is
in particular a provable $\Pi^1_1$ sentence, hence a theorem of
$\RCA$, and the first-order part of $\RCA$ is $\mathrm{I}\Sigma_1$
\cite[\S IX.1]{Simpson}.
\end{proof}

Composing with the Parsons--Mints theorem gives the finitistic
reducibility promised in the introduction.

\begin{corollary}\label{cor:pra}
$\WKL_0+\Rfns$ is $\Pi^0_2$-conservative over $\mathsf{PRA}$.
\end{corollary}

\begin{proof}
A $\Pi^0_2$ theorem of $\WKL_0+\Rfns$ is a first-order theorem,
hence provable in $\mathrm{I}\Sigma_1$ by Theorem~\ref{thm:lr};
and $\mathrm{I}\Sigma_1$ is $\Pi^0_2$-conservative over
$\mathsf{PRA}$ --- the Parsons--Mints theorem
\cite{Parsons1970,Mints1971}; see also \cite[\S IX.3]{Simpson}.
\end{proof}

The conservation extends to the weak bases, by splitting on the
single $\Pi^1_1$ sentence $\ISig$: the towers cover the models
satisfying it, and the self-satisfaction theorem of
Section~\ref{sec:wkl} covers the rest.

\begin{corollary}\label{cor:weakcons}
$\WKLs+\Rfns$ is $\Pi^1_1$-conservative over $\RCAs$.
\end{corollary}

\begin{proof}
Recall that $\ISig$ is $\Sigma^0_1$ induction with set parameters, a
single sentence, and that $\RCAs+\ISig=\RCA$. Let
$\psi=\forall X\,\theta(X)$ be a $\Pi^1_1$ sentence with
$\RCAs\nvdash\psi$, and fix a countable
$(M,\mathcal{S})\models\RCAs+\neg\psi$ with $Z\in\mathcal{S}$
witnessing $\neg\theta(Z)$. If $(M,\mathcal{S})\models\ISig$, it is
a model of $\RCA+\neg\psi$, and Corollary~\ref{cor:rfnext} extends
it to a model of $\WKL_0+\Rfns+\neg\psi$, as in the proof of
Theorem~\ref{thm:lr} --- a fortiori a model of
$\WKLs+\Rfns+\neg\psi$. If $(M,\mathcal{S})\models\neg\ISig$, extend
it, over the same first-order universe, to a model of $\WKLs$ by the
Simpson--Smith theorem relativized as in \cite[\S3]{FKWY}. The
failures $\neg\ISig$ and $\neg\theta(Z)$ are arithmetical in set
parameters from $\mathcal{S}$, so they persist, and
Theorem~\ref{thm:wkllr} gives $\Rfns$ in the extension, which is
thus a model of $\WKLs+\Rfns+\neg\psi$. Either way
$\WKLs+\Rfns+\neg\psi$ is consistent, so $\WKLs+\Rfns\nvdash\psi$.
\end{proof}

\begin{remark}\label{rem:lr}
It is worth recording what the \emph{exactness} of the coding lemma
--- the sections of the new set list the given family precisely,
rather than merely containing it --- contributes to
$\omega_1$-tower constructions like that of
Theorem~\ref{thm:towers}. Suppose the tower were built from inexact
covers, the sections of $X_\alpha$ merely containing
$\mathcal{S}_\alpha$. Countable subfamilies of the union would still
lie among the sections of single covers, so the union would still
have the covering property, and with it collection for all $L_2$
formulas by Lemma~\ref{lem:covering}(i). But the section family of
an inexact cover carries junk that need appear at no stage of the
chain. Nothing then keeps the section families of the covers closed
under unions of chains, the continuity clause of
Definition~\ref{def:covering}(ii) is not secured, and the proof of
reflection breaks down.
\end{remark}

\section{Optimality: reflection and uniformly enumerated
universes}\label{sec:optimality}

We now show that the conservation of Theorem~\ref{thm:lr} is optimal:
it stops at $\Pi^1_1$. Call a model $(\mathbb{M},\mathcal{S})$ of
$\RCAs$ \emph{uniformly enumerated} if for some $L_2$ formula
$\varphi(x,y)$, possibly with set parameters,
\[
\forall X\,\exists x\,\forall y\,(y\in X\leftrightarrow\varphi(x,y)),
\tag{$U_\varphi$}
\]
so that a single formula names every set of the model by a number.
Such models are not pathological. In $\mathrm{REC}$, the minimum
$\omega$-model of $\RCA$, consisting of the recursive sets,
$U_\varphi$ holds for a universal $\Sigma^0_1$ formula $\varphi$;
in $\mathrm{ARITH}$, the minimum $\omega$-model of $\ACA$,
consisting of the arithmetical sets, it holds for a $\Delta^1_1$
formula enumerating the arithmetical sets by iterated-jump indices
\cite[\S\S I.4, I.7]{Simpson}. And there are $\omega$-models of
$\WKL_0$ in which every set is low. Such models are easily
obtained from the Low Basis Theorem of Jockusch and Soare
\cite{JockuschSoare}: iteratively append, to the sets so far, a
low path through a low infinite binary tree, the relativized form
of the theorem keeping each step low; see \cite[\S
VIII.2]{Simpson}. In them every set is $\Delta^0_2$, so
$U_\varphi$ holds for a universal $\Sigma^0_2$ formula. Note that
for arithmetical $\varphi$ the sentence $U_\varphi$ is, of course,
$\Pi^1_1$.

\begin{theorem}\label{thm:unif}
Over $\RCAs$, for every $L_2$ formula $\varphi(x,y)$ the instance
$\Rfns_\Phi$, with $\Phi$ the subformula closure of
$\exists X\,\forall y\,(y\in X\leftrightarrow\varphi(x,y))$, refutes
$U_\varphi$. Thus reflection fails in every uniformly enumerated model;
and for arithmetical $\varphi$ the refuting instance lies in $\Rfn{1}$.
\end{theorem}

\begin{proof}
Write $\sigma(x,X)$ for $\forall y\,(y\in
X\leftrightarrow\varphi(x,y))$ and $\chi(x)$ for
$\exists X\,\sigma(x,X)$. Suppose $U_\varphi$, and apply $\Rfns_\Phi$
with the parameters of $\varphi$, obtaining $Y$ containing them among
its sections and elementary for $\chi$ and $\sigma$. The diagonal set
$D:=\{z:\langle z,z\rangle\notin Y\}$ exists by $\Delta^0_1$
comprehension and differs from every section of $Y$:
$z\in D\leftrightarrow z\notin(Y)_z$. By $U_\varphi$ there is $e$ with
$\sigma(e,D)$, so $\chi(e)$ holds, and by elementarity $\chi^Y(e)$:
some section $(Y)_z$ has $\sigma^Y(e,(Y)_z)$. Then $\sigma(e,(Y)_z)$
by elementarity for $\sigma$ (for arithmetical $\varphi$, $\sigma$ is
arithmetical and absolute outright). But then
$(Y)_z=\{y:\varphi(e,y)\}=D$, and $D$ is no section: a
contradiction. For arithmetical $\varphi$ the formula $\chi$ is
equivalent to a $\Sigma^1_1$ formula, so the instance lies in
$\Rfn{1}$.
\end{proof}

\begin{corollary}\label{cor:unif}
$\RCA+\Rfn{1}$ is not $\Sigma^1_1$-conservative over $\RCA$, and
$\WKL_0+\Rfn{1}$ is not $\Sigma^1_1$-conservative over $\WKL_0$.
\end{corollary}

\begin{proof}
Let $\varphi_1$ and $\varphi_2$ be universal $\Sigma^0_1$ and
$\Sigma^0_2$ formulas. The sentences $\neg U_{\varphi_1}$ and
$\neg U_{\varphi_2}$ are $\Sigma^1_1$, and both are provable in
$\RCA+\Rfn{1}$ by Theorem~\ref{thm:unif}. But
$\RCA\nvdash\neg U_{\varphi_1}$ and
$\WKL_0\nvdash\neg U_{\varphi_2}$, since these sentences fail in
$\mathrm{REC}$ and in the low $\omega$-models respectively.
\end{proof}

The unprovability these $\omega$-models witness is not special to the
weak bases: by Theorem~\ref{thm:acalr} and \cite[\S VII.6]{Simpson},
$\Rfn{1}$ is equivalent over $\ACA$ to $\Pi^1_1\text{-CA}_0$, so even
$\ACA$ does not prove it. 

\begin{remark}\label{rem:rc}
The $\Pi^1_2$ consequences of $\Rfns$ can also be described exactly.
For an $L_2$ formula $\varphi$ and a number $n$, using $\Rfns$ we prove
over $\RCA$ ($\WKL_0$) the
$\Pi^1_2$ sentence asserting that above any given set there is a
finite chain of $n$ coded $\omega$-models of $\RCA$ ($\WKL_0$), each a
section of the next, any two of which agree on $\varphi$ at
parameters from the lower one.
Using model-theoretic methods it is easy to show that
these sentences exhaust the $\Pi^1_2$ consequences of $\Rfns$ over the
respective base: every $\Pi^1_2$
theorem of $\RCA+\Rfns$ ($\WKL_0+\Rfns$) is provable in $\RCA$ ($\WKL_0$) together with these
sentences. Namely, if a $\Sigma^1_2$ sentence $\sigma$ is consistent with $\RCA$ ($\WKL_0$) together with these
sentences, then we first use the compactness theorem to construct a model with constants $C,K_0,K_1,\ldots$
such that $C$ witnesses the outer existential quantifier of $\sigma$
and $K_0,K_1,\ldots$ is a chain of coded $\omega$-models of $\RCA$
($\WKL_0$), each next containing the previous as a section and $C$ a
section of $K_0$, any two of which agree on every $L_2$ formula at
parameters from the lower one; then we
finish the proof by constructing a model of $\RCA+\Rfns+\sigma$ ($\WKL_0+\Rfns+\sigma$)
by restricting the second-order part to the sets that are sections of at least one $K_i$.
Note that by Corollary~\ref{cor:unif}, the bases do not prove
all of these sentences.
\end{remark}

\section{The metamathematics of the conservation}\label{sec:onecon}

The conservativity of Theorem~\ref{thm:lr} is, as a statement about
proofs, $\Pi^0_2$. But the proof just given is strongly non-finitistic:
the tower has length $\omega_1$, and the argument as written is not
formalizable in $\Ztwo$. We now bound the metamathematical cost. Write
$\ZFCm$ for $\mathsf{ZFC}$ with the power-set axiom deleted, formulated
with the collection schema rather than replacement \cite{GHJ}.

\begin{theorem}\label{thm:onecon}
$\mathsf{PRA}+\mathrm{1\text{-}Con}(\Ztwo)$ proves that $\WKL_0+\Rfns$ is
$\Pi^1_1$-conservative over $\RCA$. 
\end{theorem}

\begin{proof}
The heart of the matter is the following claim:
\emph{for every $L_2$ formula $\varphi$, the set theory
$\ZFCm+V{=}L+{}$``every set is countable'' proves the
$\Pi^1_1$-conservativity of $\WKL_0+\Rfns_\varphi$ over $\RCA$, that
is, $\forall\psi\in\Pi^1_1\,(\mathrm{Pr}_{\WKL_0+\Rfns_\varphi}(\psi)\to
\mathrm{Pr}_{\RCA}(\psi))$.} From the proof given below it will be
clear that these set-theoretic proofs can be produced primitive
recursively from $\varphi$, so that the claim itself is formalizable
in $\mathsf{PRA}$.

First let us prove that the claim yields the theorem. The conservativity
statement being proved is $\Pi^0_2$, and over $\mathsf{PRA}$ the principle
$\mathrm{1\text{-}Con}(T)$ is equivalent to uniform
$\Pi^0_2$-reflection for $T$. So $\mathsf{PRA}$ together with the
$1$-consistency of the set theory of the claim proves each
conservativity statement of the claim to be \emph{true}. Then,
given a proof of a $\Pi^1_1$ sentence $\psi$ in $\WKL_0+\Rfns$, we
extract from it the finitely many reflection instances used; by the
amalgamation following Definition~\ref{def:lr}, a single instance
$\Rfns_\varphi$ implies them all, so $\WKL_0+\Rfns_\varphi\vdash\psi$,
with $\varphi$ found primitive recursively from the proof. Applying
the conservativity for that $\varphi$ we obtain
$\mathrm{Pr}_{\RCA}(\psi)$. This is the conservation theorem.

Next, the $1$-consistency of $\ZFCm+V{=}L+{}$``every set is
countable'' is implied by $\mathrm{1\text{-}Con}(\Ztwo)$: $\Ztwo$ proves every axiom of the set
theory of the claim to hold in the constructible sets, hereditarily
countable and coded, say, by well-founded extensional trees on $\N$.
A complete proof of this interpretation result is given by Kanovei and
Lyubetsky \cite{KanoveiLyubetsky}, who also trace its history. The
result has been known in essence since Kreisel's survey
\cite{Kreisel1968}; it was published by Apt and Marek
\cite[Theorem~5.5]{AptMarek} in a version with the countable choice
scheme on the arithmetical side, and by Marek
\cite{Marek1978} without it. This provability is
itself verifiable in $\mathsf{PRA}$: the relativizing proof
translation is primitive recursive and leaves arithmetical sentences
fixed, so $\mathsf{PRA}$ proves that $\mathrm{1\text{-}Con}(\Ztwo)$
implies the $1$-consistency of the set theory.

\emph{Proof of the claim.} We reason in
$\ZFCm+V{=}L+{}$``every set is countable'' and check that
the proof of Theorem~\ref{thm:lr} with $\Rfns_\varphi$ instead of
$\Rfns$ goes through, pausing at the steps that need attention.
It suffices to verify that the proof of the variant of Corollary~\ref{cor:rfnext}
with $\Rfns_\varphi$ instead of $\Rfns$ can be carried out.
The generic listing of Lemma~\ref{lem:codingns}, the Harrington extensions of
Lemma~\ref{lem:wklext}, and the successor
and limit steps of Theorem~\ref{thm:towers} are ordinary mathematics
about countable objects and formalize in $\ZFCm$ as they stand. The following
steps require more attention. First we need to adapt the proof of
Theorem~\ref{thm:towers}, and the adapted proof will produce a model $(M,\mathcal{S})$
with a class-sized second-order part $\mathcal{S}$
(we do this below).
Then we use the argument of Lemma~\ref{lem:covering}(ii) to show that $\Rfns_\varphi$
holds in $(M,\mathcal{S})$ (we also do this below). This, however, does not
yet deliver Corollary~\ref{cor:rfnext}, since we want to get a set model
instead of a class model, for which we lack a full
satisfaction class. What we know at this point is that the
$(M,\mathcal{S})$-relativizations of $\Rfns_\varphi$ and of all the axioms of
$\mathsf{WKL}_0$ hold (we fix some finite axiomatization of $\mathsf{WKL}_0$). Then
we use transitive model reflection --- provable in $\ZFCm+V{=}L$ by
the usual proof of the L\'evy--Montague reflection theorem in
$\mathsf{ZFC}$, with the hierarchy $\langle
L_\alpha\rangle_{\alpha\in\mathrm{Ord}}$ in place of $\langle
V_\alpha\rangle_{\alpha\in\mathrm{Ord}}$ --- to reflect, to a
transitive set, the previous sentence together with the assertion that
$(M,\mathcal{S})$ extends the starting (set) model $(M,\mathcal{S}_0)$.
From the external point of view this gives us a set-sized $(M,\mathcal{S}')\models \mathsf{WKL}_0+\Rfns_\varphi$
extending $(M,\mathcal{S}_0)$, which allows
us to obtain Corollary~\ref{cor:rfnext} and thus finish the proof.

It remains to discharge the two debts incurred above. The first is
the adapted proof of Theorem~\ref{thm:towers}: we run its
construction with a class model as the result. Every set is countable, so there is no
$\omega_1$ to stop at: the tower runs along \emph{all} ordinals. A
class-length recursion must make its choices uniformly definably, and
this is what $V{=}L$ is for: at every point of the construction take
the $<_L$-least object with the required property (the $<_L$-least
enumeration of the dense classes, the filter obtained by always
passing to the $<_L$-least condition meeting the next class, the
$<_L$-least Harrington extensions). Thus we obtain a
definable, increasing, continuous class sequence
$\langle\mathcal{S}_\alpha:\alpha\in\mathrm{Ord}\rangle$ of countable
families with $(M,\mathcal{S}_\alpha)\models\WKL_0$ for $\alpha>0$
and exact covers
$X_\alpha\in\mathcal{S}_{\alpha+1}$. We write $\mathcal{S}$ for
the union, a definable class family; its continuous covering
property for $\mathcal{I}=\{X_\alpha:\alpha\in\mathrm{Ord}\}$
is verified as in Theorem~\ref{thm:towers}.

The second debt is reflection in the class model: we check that the
proof of Lemma~\ref{lem:covering}(ii) works
without changes for $(M,\mathcal{S})$ to show that
$\Rfns_\varphi$ holds in it. The family $\Phi$ of subformulas of
$\varphi$ used in that proof is now a fixed finite family.
We can therefore write a fixed first-order formula of set theory defining
the $<_L$-least suitable $X_{n+1}$ from the previous $X$'s, and use
the replacement schema to assemble the $X_n$'s into a set sequence
$\langle X_n:n\in\omega\rangle$ --- this is where the full strength
of the schema comes into play. The rest of the proof, the limit over
the chain and the induction over the subformulas, reads as before.
\end{proof}

What underwrites the status of $\WKL_0$ as a partial realization of
Hilbert's program is not merely that it is $\Pi^0_2$-conservative over
$\mathsf{PRA}$, but that this conservation is established by
\emph{finitary} means.\footnote{Friedman's original argument was
model-theoretic, but by now several genuinely finitistic proofs are
known: Sieg \cite{Sieg1985} gave a primitive recursive transformation
converting any $\WKL_0$-proof of a $\Pi^0_2$ sentence into a
$\mathsf{PRA}$-proof of the same sentence, by cut elimination, Herbrand
analysis, and majorization; H\'ajek \cite{Hajek1993} interpreted
$\WKL_0$ in $\mathsf{I}\Sigma_1$ by formalizing the low basis theorem,
which yields the theorem when composed with Parsons' conservation of
$\mathsf{I}\Sigma_1$ over $\mathsf{PRA}$; Avigad \cite{Avigad1996}
formalized Harrington's forcing argument within $\RCA$ itself, giving a
polynomial-length translation of proofs of $\Pi^1_1$ theorems of
$\WKL_0$ into $\RCA$; and Kohlenbach \cite{Kohlenbach1992} eliminated
weak K\"onig's lemma by G\"odel's functional interpretation combined
with majorization.} The conservativity of Theorem~A has
no such pedigree as it stands. Our proof passes through the
$\omega_1$-tower, a genuinely uncountable construction, and
Theorem~\ref{thm:onecon} compresses it into
$\mathsf{PRA}+\mathrm{1\text{-}Con}(\Ztwo)$, which is still a long way
from finitism.

\begin{problem}\label{prob:finitism}
Is the conservativity of Theorem~A finitistically provable?
\end{problem}

The problem has a quantitative companion. A proof of a conservation
theorem bounds the associated \emph{speed-up}: the growth in length
between a proof of a $\Pi^1_1$ sentence in the extended theory and
the shortest proof in the base. The usual situation is a polynomial
proof transformation --- Avigad's translation \cite{Avigad1996}
bounds the speed-up of $\WKL_0$ over $\RCA$ polynomially --- or a
hyperexponential speed-up, as with the conservativity of $\ACA$
over first-order Peano arithmetic \cite{Pudlak}.
Theorem~\ref{thm:onecon}, by contrast, bounds the speed-up of
$\WKL_0+\Rfns$ only by a function provably total in
$\mathsf{PRA}+\mathrm{1\text{-}Con}(\Ztwo)$.

\begin{problem}\label{prob:speedup}
What is the speed-up of $\WKL_0+\Rfns$ over $\WKL_0$ --- equivalently,
by \cite{Avigad1996}, over $\RCA$ --- for proofs of $\Pi^1_1$
sentences? In particular, is it bounded by a primitive recursive
function?
\end{problem}

Theorem~\ref{thm:onecon} and Problems~\ref{prob:finitism}
and~\ref{prob:speedup} have a precedent in the metamathematics of
Takeuti's fundamental conjecture, cut elimination for the
second-order sequent calculus: proved by Tait \cite{Tait1966}, it
is a $\Pi^0_2$ theorem equivalent over $\mathsf{PRA}$ to
$\mathrm{1\text{-}Con}(\Ztwo)$, as is the strong normalization of
System~F, through which Girard reproved it \cite{Girard1971}; see
\cite{Takeuti1987,GirardLafontTaylor}. The conjecture has a
conservation reading: comprehension enters a proof of a
first-order sequent only through cuts, and a cut-free proof of
such a sequent is first-order throughout, so cut elimination makes
the expansion of any first-order theory by full second-order logic
conservative over that theory. The metamathematics is analogous to
that of Theorem~\ref{thm:onecon}: each instance of cut elimination
is provable in $\Ztwo$, and the uniform $\Pi^0_2$ statement
follows by uniform $\Pi^0_2$-reflection for $\Ztwo$, which over
$\mathsf{PRA}$ is exactly $\mathrm{1\text{-}Con}(\Ztwo)$. But here
the cost is known to be necessary: the speed-up exceeds every
$\Ztwo$-provably total computable function. Take Robinson's
$\mathsf{Q}$ for the first-order theory. The second-order
expansion interprets $\Ztwo$ on a definable cut, and so proves,
for each Turing machine that $\Ztwo$ proves total, the $\Sigma_1$
sentences asserting its termination on particular inputs, by
proofs of length elementary in the input; a $\mathsf{Q}$-proof of
such a sentence yields a witness bounded primitive recursively in
its length, and so is, up to a primitive recursive shift, at least
as long as the computation it certifies. These methods are not
directly applicable to $\Rfns$, and they leave multiple
possibilities open: Theorem~A may follow the precedent of
$\WKL_0$, whose model-theoretic conservation proof was compressed
into finitary form after the fact, or that of System~F, where the
speed-up is real and $\mathrm{1\text{-}Con}(\Ztwo)$ is the price,
or perhaps a more conventional intermediate scenario, like the
conservativity of $\mathsf{ACA}_0$ over $\mathsf{PA}$.

\end{document}